\documentclass[11pt]{amsart}
\usepackage{amsmath, amssymb, amsthm}
\usepackage{enumerate}
\usepackage[colorlinks=true,allcolors=blue]{hyperref}

\usepackage{tikz}
\usetikzlibrary{arrows}
\usetikzlibrary{calc}
\usetikzlibrary{cd}

\newcommand{\C}{\mathbb{C}}

\newcommand{\Z}{\mathbb{Z}}
\newcommand{\N}{\mathbb{N}}
\newcommand{\Q}{\mathbb{Q}}
\newcommand{\K}{\mathbb{K}}

\newcommand{\Aut}[1]{\text{Aut}(#1)}
\newcommand{\id}[1]{\operatorname{id}_{#1}}
\newcommand{\obj}[1]{\operatorname{obj}(#1)}
\newcommand{\EV}[1]{\operatorname{EV}(#1)}

\newcommand{\cE}{\mathcal{E}}

\newcommand{\cG}{\mathcal{G}}
\newcommand{\Eig}[2]{\text{Eig}(#1,#2)}
\newcommand{\Endo}[1]{\text{End}(#1)}

\newcommand{\Vfin}{\mathcal{V}^{{\rm fin}}_{\C}}
\newcommand{\Viso}{\mathcal{V}^{{\rm iso}}_{\C}}
\newcommand{\BUpl}{\mathcal{B}U_{\oplus}}
\newcommand{\BUte}[1]{\mathcal{B}U_{\otimes}^{#1}}
\newcommand{\Zloc}[1]{\mathbb{Z}\!\left[\tfrac{1}{#1}\right]}
\newcommand{\BUloc}[1]{BU_{\otimes}\!\left[\tfrac{1}{#1}\right]}
\newcommand{\funeq}{\sim_f}
\newcommand{\moneq}{\sim_{\otimes}}

\newcommand{\I}{\mathcal{I}}

\newcommand{\Bte}[1]{\mathcal{B}^{#1}_{\otimes}}
\newcommand{\Proj}[2]{\mathcal{P}_{#1}(#2)}

\newcommand*\circled[1]{\tikz[baseline=(char.base)]{
            \node[shape=circle,draw,inner sep=0.7pt] (char) {#1};}}
            
\newtheorem{theorem}{Theorem}[section]
\newtheorem{lemma}[theorem]{Lemma}
\newtheorem{corollary}[theorem]{Corollary}
\theoremstyle{definition}
\newtheorem{definition}[theorem]{Definition}
\newtheorem{remark}[theorem]{Remark}

\DeclareMathOperator{\hocolim}{hocolim}
\DeclareMathOperator{\ch}{ch}

\makeatletter
\newcommand{\extp}{\@ifnextchar^\@extp{\@extp^{\,}}}
\def\@extp^#1{\mathop{\bigwedge\nolimits^{\!#1}}}
\makeatother

\begin{document}
\title{Exponential Functors, $R$-matrices and twists}
\author{Ulrich Pennig}
\begin{abstract}
	In this paper we show that each polynomial exponential functor on complex finite-dimensional inner product spaces is defined up to equivalence of monoidal functors by an involutive solution to the Yang-Baxter equation (an involutive $R$-matrix), which determines an extremal character on $S_{\infty}$. These characters are classified by Thoma parameters, and Thoma parameters resulting from polynomial exponential functors are of a special kind. Moreover, we show that each $R$-matrix with Thoma parameters of this kind yield a corresponding polynomial exponential functor.
	
	In the second part of the paper we use these functors to construct a higher twist over $SU(n)$ for a localisation of $K$-theory that generalises the one classified by the basic gerbe. We compute the indecomposable part of the rational characteristic classes of these twists in terms of the Thoma parameters of their $R$-matrices.
\end{abstract}

\maketitle

\section{Introduction}
This paper consists of two parts: In the first part we study exponential functors on the category $\Vfin$ of complex finite-dimensional inner product spaces. Such a functor is defined to be monoidal with respect to the direct sum on the domain and the tensor product on the codomain, i.e.\ it comes with a natural isomorphism 
\[	
	\tau_{V,W} \colon F(V \oplus W) \to F(V) \otimes F(W)
\]
and a corresponding unit isomorphism $\iota \colon F(0) \to \C$, such that associativity and unitality conditions hold (see Def.~\ref{def:exp_functor}). A classical example of an exponential functor $F \colon \Vfin \to \Vfin$ is the full exterior algebra $V \mapsto \extp^*V$, which is also a polynomial functor in the sense of \cite[App.~A]{book:Macdonald}. If the domain is restricted to the subgroupoid $\Viso \subset \Vfin$, then the top exterior power $V \mapsto \extp^{\rm top}V$ provides another example. 

We employ the theory of polynomial functors developed by Macdonald in \cite[App.~A]{book:Macdonald} to show that isomorphism classes of polynomial exponential functors $\Vfin \to \Vfin$ correspond to equivalence classes of involutive solutions to the Yang-Baxter equation. On $W \in \obj{\Vfin}$ this equation asks for linear transformations $R \colon W^{\otimes 2} \to W^{\otimes 2}$  that satisfy
\[
	(R \otimes \id{W})(\id{W} \otimes R)(R \otimes \id{W}) = (\id{W} \otimes R)(R \otimes \id{W})(\id{W} \otimes R)
\]
on $W^{\otimes 3}$. Solutions to this equation are called $R$-matrices. If they are involutive (i.e.\ if they additionally satisfy $R^2 = \id{W}$), then they give rise to representations of all symmetric groups $S_n$ and to extremal characters of the infinite symmetric group $S_{\infty}$. The group $S_{\infty}$ has infinite conjugacy classes and is therefore not type~I. Nevertheless, its extremal characters were fully classified by Thoma in \cite{paper:Thoma} (see also \cite[Sec.~2.3]{paper:Okounkov}) by the Thoma parameters $((\alpha_i)_{i \in \N},(\beta_j)_{j \in \N})$. These are two sequences of positive real numbers lying in the Thoma simplex $\mathbb{T}$ (see Sec.~\ref{sec:Thoma_par}). 

The characters and their Thoma parameters corresponding to $R$-matrices were identified in \cite{preprint:LechnerPennigWood} using subfactor theory. In the current paper we show that each polynomial exponential functor $F$ gives rise to an $R$-matrix acting on the linearisation of $F$, which characterises it up to isomorphism. We also show that the tensor product of exponential functors is linked to the box sum of $R$-matrices defined in \cite[Def.~4.1]{preprint:LechnerPennigWood}. It turns out to be more convenient to work with Thoma parameters rescaled by the dimension of the representation as in \cite[(4.13)]{preprint:LechnerPennigWood}. We prove:
\begin{theorem}
	Let $F \colon \Vfin \to \Vfin$ be a polynomial exponential functor. Let $(F_n)_{n \in \N}$ be its homogeneous summands. Let $W = F_1(\C)$. Then there is an involutive solution $R \colon W^{\otimes 2} \to W^{\otimes 2}$ to the Yang-Baxter equation 	associated to $F$ called the $R$-matrix of $F$. It has the following properties: 
	\begin{enumerate}[a)]
		\item The $R$-matrix determines $F$ up to natural monoidal equivalence (Def.~\ref{def:iso_and_str_iso}) of exponential functors (Lem.~\ref{lem:F_R_moneq}).
		\item The rescaled Thoma parameters of $R$ are of the form $(0, b)$ for a finite sequence of non-negative integers $b = (b_1, \dots, b_m)$ (Lemma~\ref{lem:Thoma_parameters}).
		\item For each involutive $R$-matrix $R \colon W^{\otimes 2} \to W^{\otimes 2}$ with rescaled Thoma parameters $(0,(b_1, \dots, b_m))$ for non-negative integers $b_i$, there exists a polynomial exponential functor $F^R$ such that its $R$-matrix is $R$ (Thm.~\ref{thm:classification_part_2}).
		\item The rescaled Thoma parameters determine $F$ up to natural equivalence of functors (i.e.\ neglecting the monoidal transformation, Def.~\ref{def:iso_and_str_iso}) and for any choice of parameters as in b) there exists an exponential functor with these parameters (Lem.~\ref{lem:F_R_equivalences}, \cite[Thm.~4.8]{preprint:LechnerPennigWood} and Thm.~\ref{thm:classification_part_1}). 
		\item The $R$-matrix of the tensor product $F \otimes G$ of two exponential functors $F$ and $G$ with associated $R$-matrices $R$ and $S$, respectively, is given by the box-sum $R \boxplus S$ (Thm.~\ref{thm:box_sum}). 
	\end{enumerate}
\end{theorem}

Our motivation for viewing exponential functors from this new angle will become clear in the second part of the paper, where they yield the input datum for a construction of higher twists over $SU(n)$ in a localisation of $K$-theory generalising the twist given by the basic gerbe \cite{paper:Meinrenken, paper:MurrayStevenson}. 

Twisted $K$-theory first appeared in \cite{paper:DonovanKaroubi} and is the $K$-theoretic analogue of ordinary cohomology with local coefficients. It associates $K$-groups to a pair consisting of a topological space $X$ and an extra structure over $X$, called a twist, which plays the role of a coefficient bundle. Up to isomorphism the twists in \cite{paper:DonovanKaroubi} were classified by $H^1(X,\Z/2\Z) \times {\rm Tor}(H^3(X,\Z))$. This was generalised in \cite{paper:AtiyahSegal2004, paper:AtiyahSegal2006} to include non-torsion elements in $H^3(X,\Z)$. These twists have several geometric interpretations, e.g.\ in terms of bundles of compact operators as in \cite{paper:DonovanKaroubi, paper:AtiyahSegal2004, paper:AtiyahSegal2006} or in terms of (bundle) gerbes \cite{paper:BCMMS, paper:MurrayStevenson, paper:Meinrenken, paper:Waldorf}, which can be viewed as generalisations of line bundles.

Twists over Lie groups gained increasing importance in the subject: By results of Freed, Hopkins and Teleman the equivariant twisted $K$-theory of the group with respect to multiples of the basic gerbe is related to the representation theory of the corresponding free loop group \cite{paper:FreedHopkinsTelemanI, paper:FreedHopkinsTelemanII, paper:FreedHopkinsTelemanIII}. This result is also closely connected to the modular invariants of rational conformal field theories associated to the loop groups \cite{paper:EvansGannon}.

Looking at the case of non-equivariant twists from a homotopy theoretic viewpoint the appearance of $H^1(X,\Z/2\Z) \times H^3(X,\Z)$ is not a surprise. In fact, any $E_{\infty}$-ring spectrum $R$ has a spectrum of units $gl_1(R)$ and the first delooping $BGL_1(R)$ of its underlying infinite loop space classifies the twists of $R$-theory \cite{paper:AndoBlumbergGepnerHopkinsRezk}. For the complex $K$-theory spectrum $KU$ we obtain an equivalence of spaces
\[
	BGL_1(KU) \simeq K(\Z/2\Z,1) \times BBU_{\otimes}\ ,
\]
where the second factor is the delooping of the classifying space of virtual one-dimensional vector bundles with respect to the tensor product. The twists described above are the ones that factor through the natural map from $K(\Z/2\Z,1) \times BBU(1)$ to $BGL_1(KU)$. For example, the basic gerbe represents a generator of $H^3(SU(n),\Z) \cong [SU(n), BBU(1)] \cong \Z$. 

The goal in the second part of this paper is to study a generalisation of this twist. The input datum is an exponential functor $F \colon \Viso \to \Viso$ on the groupoid $\Viso \subset \Vfin$ containing only the unitary isomorphisms as morphisms. The functor $F$ induces an $A_{\infty}$-map $BU_{\oplus} \to \BUloc{d}$ for $d = \dim(F(\C))$, which deloops to a map $BBU_{\oplus} \to B\BUloc{d}$. This result will be proven via monoidal categories in Sec.~\ref{sec:twists_via_mon_cats}. Another model for $B\BUloc{d}$ using $\I$-spaces is given in an appendix. The twists $\tau^n_F$ over $SU(n)$ are then defined as the following compositions
\[
	\tau^n_F \colon SU(n) \to SU \simeq BBU_{\oplus} \to B\BUloc{d}\ .
\]
where $SU(n) \to SU$ is the colimit inclusion and $SU \simeq BBU_{\oplus}$ is a consequence of Bott periodicity. The case of the basic gerbe corresponds to the choice $F = \extp^{\rm top}$. For this functor $d = \dim(F(\C)) = 1$ and we will see in Thm.~\ref{thm:class_of_classical_twist} that it represents a generator of $H^3(SU(n),\Z) \subset H^3(SU(n),\Q)$. This should be compared with the construction of the basic gerbe over $U(n)$ given in \cite{paper:MurrayStevenson}. A similar approach to twists of $K_G(X)[[t]]$ based on exponential maps has been used in \cite[(3.12)]{paper:Teleman}. Instead of an extension by a formal variable $t$ the price to be paid in our approach is that the result will be a twist of the localisation $KU\!\left[\tfrac{1}{d}\right]$. 

Even though the generalised cohomology theory defined by the infinite loop space $BBU_{\otimes}$ has some subtle features over the integers, its rationalisation $\left(BBU_{\otimes}\right)_{\Q}$ is well-understood. The logarithm composed with the first delooping of the Chern character induces an  equivalence
\[
	\left(BBU_{\otimes}\right)_{\Q} \simeq \prod_{k \in \N} K(\Q,2k+1)\ .
\]
If we apply the above equivalence to the composition of $\tau^n_F$ with the canonical map $B\BUloc{d} \to \left(BBU_{\otimes}\right)_{\Q}$ we obtain odd degree rational characteristic classes
\[
	\delta^{F,n} = \delta^{F,n}_3 + \delta^{F,n}_5 + \dots \quad \in H^{\text{odd}}(SU(n),\Q) \cong \Lambda_{\Q}^*[a_3,a_5,\dots, a_{2n-1}]\ .
\]
Suppose that $\delta^{F,n} = \sum_{i=1}^{n-1} \kappa_i a_{2i+1} + r$, where $r$ denotes the remainder consisting of all decomposable terms. If $F \colon \Vfin \to \Vfin$ is a polynomial exponential functor, then $\kappa_i$ can be expressed in terms of the Thoma parameters $(b_1, \dots, b_m)$ associated to the $R$-matrix of $F$. More precisely, it is shown in Thm.~\ref{thm:char_classes} that $\kappa_i$ is the $i$th coefficient in the Taylor expansion of the function
\[
	\kappa(x) = \sum_{j=1}^m \log \left( \frac{1+b_j e^x}{1+b_j} \right)\ .
\]
In particular, it follows that $\tau^n_F$ in general does not factor through the classifying space of classical twists and therefore is indeed a higher twists. 

The twisted $K$-groups for $\tau^n_F \colon SU(n) \to B\BUloc{d}$ can now be obtained by applying \cite[Def.~2.27]{paper:AndoBlumbergGepnerHopkinsRezkII} to $R = KU\left[\tfrac{1}{d}\right]$. In joint work with Dadarlat the author also developed an operator algebraic model for twists of this type in \cite{paper:DadarlatPennigI, paper:DadarlatPennigII, paper:Pennig}. To explain how the current paper fits together with these ideas and how exponential functors come into the game we will end with an outlook on how to represent the twists $\tau^n_F$ via Fell bundles. This approach is a generalisation of the construction in \cite[Rem.~4.4]{paper:Pennig} used to describe twists over suspensions. Fell bundles originated in the theory of $C^*$-dynamical systems and can be seen as generalised bundle gerbes with $C^*$-correspondences replacing lines. A full description of this construction can be found in \cite{paper:EvansPennigFellBundles}. One of its appealing features is that it is equivariant with respect to the adjoint action of $SU(n)$ on itself and therefore provides an equivariant higher twist over $SU(n)$ generalising the one used in \cite{paper:FreedHopkinsTelemanI, paper:FreedHopkinsTelemanII, paper:FreedHopkinsTelemanIII}. 

\vspace{2mm}
\paragraph{\bf Acknowledgements} The author would like to thank David Evans for several enlightening discussions about twisted $K$-theory of compact Lie groups. Part of this work was completed while the author was staying at the Newton Institute for the programme ``Operator algebras: subfactors and their applications''. The author would like to thank the organisers of the programme for the invitation and the Newton Institute for its hospitality. His research was supported in part by the EPSRC grant EP/K032208/1. 

\section{Exponential Functors on $\Vfin$}
Let $\Vfin$ be the category of finite-dimensional complex inner product vector spaces and linear maps as morphisms. This is a monoidal category in two ways: The direct sum $\oplus$ with the zero vector space as its unit object provides the first structure, the tensor product $\otimes$  and $\C$ as its unit object the second. Note in particular that both, the direct sum and the tensor product, can be equipped with inner products in a canonical way. Both monoidal structures can be extended to symmetric monoidal ones using the canonical isomorphisms $V \oplus W \to W \oplus V$ and $V \otimes W \to W \otimes V$ respectively. However, we will ignore these for now. 
\begin{definition} \label{def:exp_functor}
	A triple $(F, \tau, \iota)$ consisting of a functor $F \colon \Vfin \to \Vfin$ and natural unitary isomorphisms $\tau_{V,W} \colon F(V \oplus W) \to F(V) \otimes F(W)$ and $\iota \colon F(0) \to \C$ is called a \emph{(unitary) exponential functor} if the following conditions hold:
	\begin{enumerate}[a)]
		\item $F$ preserves adjoints, i.e.\ for all $f \colon V \to W$ we have $F(f^*) = F(f)^*$.
		\item The transformation $\tau$ is associative in the sense that the following diagram commutes:
		\begin{center}
		\begin{tikzcd}[column sep=2.4cm]
			F(V \oplus W \oplus X) \ar[r,"\tau_{V, W \oplus X}"] \ar[d,"\tau_{V \oplus W,X}"] & F(V) \otimes F(W \oplus X) \ar[d,"\id{F(V)} \otimes \tau_{W,X}"] \\
			F(V \oplus W) \otimes F(X) \ar[r,"\tau_{V,W} \otimes \id{F(X)}"] & F(V) \otimes F(W) \otimes F(X)
		\end{tikzcd}
		\end{center}
		where we neglected the associators in $(\Vfin,\oplus)$ and $(\Vfin,\otimes)$ in the notation.
		\item The transformation $\iota$ makes the following diagrams commute
		\begin{center}
		\begin{tikzcd}[column sep=1cm]
			F(V \oplus 0) \ar[d] \ar[r,"\tau_{V,0}"] & F(V) \otimes F(0) \ar[d,"\id{F(V)} \otimes \iota"] \\ 
			F(V) & \ar[l] F(V) \otimes \C
		\end{tikzcd}
		\begin{tikzcd}[column sep=1cm]
			F(0 \oplus V) \ar[d] \ar[r,"\tau_{0,V}"] & F(0) \otimes F(V) \ar[d," \iota \otimes \id{F(V)}"] \\ 
			F(V) & \ar[l] \C \otimes F(V)
		\end{tikzcd}
		\end{center}
		in which the unlabelled horizontal arrows are the canonical isomorphisms and the unlabelled vertical isomorphisms are the ones induced by the canonical maps $V \oplus 0 \to V$ and $0 \oplus V \to V$.
	\end{enumerate}
	Equivalently, $F$ is a unitary monoidal functor 
	\(
		F \colon (\Vfin,\oplus) \to (\Vfin,\otimes).
	\)
\end{definition}

The next definition introduces two ways of comparing exponential functors, both of which will become important in the following. 

\begin{definition} \label{def:iso_and_str_iso}
Let $F, G \colon \Vfin \to \Vfin$ be exponential functors. We say that $F$ and $G$ are \emph{naturally (unitarily) isomorphic as functors} if there exists a natural equivalence given by a unitary isomorphism $v(X) \colon F(X) \to G(X)$. This will be denoted by $F \funeq G$

We will say that $F$ and $G$ are \emph{monoidally isomorphic} or \emph{naturally isomorphic as monoidal functors} if there exists a natural equivalence $v(X)$ as above which additionally makes the following diagram commute:
	\[
		\begin{tikzcd}[column sep=2cm]
			F(V \oplus W) \ar[r,"v(V \oplus W)"] \ar[d,"\tau^F_{V,W}" left] & G(V \oplus W) \ar[d,"\tau^G_{V,W}"]\\
			F(V) \otimes F(W) \ar[r,"v(V) \otimes v(W)" below] & G(V) \otimes G(W) 
		\end{tikzcd}
	\]
	We will denote this by $F \moneq G$.
\end{definition}

The following definition is taken from \cite[Appendix~A, Def.~1.1]{book:Macdonald}.

\begin{definition}
	A functor $F \colon \Vfin \to \Vfin$ is called \emph{polynomial} if for each pair of vector spaces $V,W$ and $n$-tuple of linear maps $f_1, \dots, f_n \colon V \to W$ the linear transformation
	\[
		F(\lambda_1f_1 + \dots + \lambda_nf_n) \colon F(V) \to F(W)
	\]
	is a polynomial in the variables $\lambda_1, \dots, \lambda_n \in \C$ with coefficients in the vector space of linear maps $\hom(F(V),F(W))$ (depending on $f_1, \dots, f_n$). If $F(\lambda_1f_1 + \dots + \lambda_nf_n)$ is homogeneous of degree $n$ for all choices of $f_1, \dots, f_n$, then $F$ is called \emph{homogeneous} of degree $n$ as well.
\end{definition}

It is shown in \cite[Appendix~A, Sec.~2]{book:Macdonald} that any polynomial functor decomposes into a direct sum of homogeneous functors. We briefly recall how this is done: By assumption the image of the multiplication operator $\lambda\,\id{V} \colon V \to V$ under $F$ can be decomposed as follows
\[
	F(\lambda\,\id{V}) = \sum_{n \in \N_0} u_n(V) \lambda^n
\]
with only finitely many nonzero coefficients $u_n(V)$. The functor properties imply that the linear maps $u_n(V) \in \Endo{F(V)}$ are idempotents, which add up to $\id{F(V)}$ and satisfy $u_n(V) u_m(V) = 0$ for $n \neq m$. Let $F_n(V)$ be the image of $u_n(V)$. Then we have a direct sum decomposition 
\[
	F(V) = \bigoplus_{n \in \N_0} F_n(V)\ .
\]
Let $f \colon V \to W$ be a linear transformation. Linearity implies $f \circ (\lambda\, \id{V}) = (\lambda\,\id{W}) \circ f$, which yields $F(f)u_n(V) = u_n(W)F(f)$. Thus, the association $V \mapsto F(V)$ restricts to a collection of functors $V \mapsto F_n(V)$, in which $F_n$ is homogeneous of degree $n$, and this collection satisfies
\[
	F \cong \bigoplus_{n \in \N_0} F_n\ ,
\]
where the natural isomorphism is induced by the inclusion and projection maps $F_n(V) \to F(V)$ and $F(V) \to F_n(V)$.

If $F$ preserves adjoints, then the decomposition of $F$ into its homogeneous components $F_n$ is orthogonal, since $F(f^*) = F(f)^*$ implies 
\[
	F(\lambda\,\id{V}) = F(\overline{\lambda}\,\id{V})^* = \sum_{n \in \N_0} \left(u_n(V) \overline{\lambda}^n\right)^* = \sum_{n \in \N_0} u_n(V)^* \lambda^n\ ,
\]
which yields that $u_n(V) = u_n(V)^* = u_n(V)^2$ are orthogonal projections for all $n \in \N_0$ and $V \in \obj{\Vfin}$.

\begin{lemma} \label{lem:exp-homogeneous}
	Let $F \colon \Vfin \to \Vfin$ be a polynomial unitary exponential functor and let $(F_n)_{n \in \N_0}$ be its decomposition into homogeneous components as explained above. Then the isomorphisms $\tau_{V,W}$ restrict to natural unitary isomorphisms 
	\[
		\tau_{V,W}^n \colon F_n(V \oplus W) \to \bigoplus_{i+j = n} F_i(V) \otimes F_j(W)\ .
	\]
\end{lemma}

\begin{proof}
Note that $\lambda\,\id{V \oplus W} = \lambda\,\id{V} \oplus \lambda\,\id{W}$. Since $\tau_{V,W}$ is natural in $V$ and $W$, the following diagram commutes:
\begin{center}
	\begin{tikzcd}
		F(V \oplus W) \ar[d,"F(\lambda\,\id{V} \oplus \lambda\,\id{W})" left] \ar[r,"\tau_{V,W}"] & F(V) \otimes F(W) \ar[d,"F(\lambda\,\id{V}) \otimes F(\lambda\,\id{V})"] \\
		F(V \oplus W) \ar[r,"\tau_{V,W}"] & F(V) \otimes F(W)	
	\end{tikzcd}
\end{center}
Comparing the coefficients of the resulting polynomials we obtain 
\[
	\tau_{V,W} \circ u_n(V \oplus W) = \sum_{i+j = n} (u_i(V) \otimes u_j(W)) \circ  \tau_{V,W}
\] where $(u_k(X))_{k \in \N_0}$ are the idempotents inducing the direct sum decomposition of $F(X)$ as above. Observe that $(u_{i_1}(V) \otimes u_{j_1}(W))(u_{i_2}(V) \otimes u_{j_2}(W)) = 0$ unless $i_1 = i_2$ and $j_1 = j_2$. Since $F_k(X)$ was defined to be the image of $u_k(X)$, this implies that $\tau_{V,W}^n$ restricts to a well-defined homomorphism $\tau_{V,W}^n \colon F_n(V \oplus W) \to \bigoplus_{i+j = n} F_i(V) \otimes F_j(W)$. A similar argument shows that $\tau_{V,W}^{-1}$ restricts to an inverse of $\tau_{V,W}^n$ proving that the latter is an isomorphism. Since $\tau_{V,W}^n$ is a restriction of a unitary map, it preserves inner products and hence is a unitary isomorphism itself. 
\end{proof}

Recall the definition of the linearization from \cite[Appendix~A, Sec.~3]{book:Macdonald}. Let $F_n \colon \Vfin \to \Vfin$ be a homogeneous polynomial functor of degree $n$. Let $V = V_1 \oplus \dots \oplus V_n$ and $M_{\lambda_1, \dots, \lambda_n} \colon V \to V$ be the linear map that multiplies by $\lambda_i$ on $V_i$. Using a similar argument as above, we obtain a direct sum decomposition
\[
	F_n(V_1 \oplus \dots \oplus V_n) = \bigoplus_{i_1 + \dots + i_n = n} F_{i_1, \dots, i_n}(V_1, \dots, V_n)\ ,
\]
where $F_{i_1, \dots, i_n}(V_1, \dots, V_n)$ is fixed by the property that $F_n(M_{\lambda_1, \dots, \lambda_n})$ acts on it by multiplication by $\lambda_1^{i_1}\cdots \lambda_n^{i_n}$. In particular, the linearization of $F_n$ is given by the functor $L_{F_n}(V_1, \dots, V_n) = F_{1, \dots, 1}(V_1, \dots, V_n)$, i.e.\ it is the natural direct summand of $F_n(V_1 \oplus \dots \oplus V_n)$, on which $F_n(M_{\lambda_1, \dots, \lambda_n})$ acts via multiplication by $\lambda_1 \cdots \lambda_n$. 

\begin{lemma} \label{lem:linearization}
	Let $F \colon \Vfin \to \Vfin$ be a polynomial exponential functor and let $(F_n)_{n \in \N_0}$ be its decomposition into homogeneous components. For $n \geq 1$ the linearization of $F_n$ is naturally isomorphic to the functor 
	\[
		L_n(V_1, \dots, V_n) = F_1(V_1) \otimes \dots \otimes F_1(V_n)\ .
	\]
\end{lemma}

\begin{proof}
	By Lemma~\ref{lem:exp-homogeneous} we can decompose $F_n(V_1 \oplus \dots \oplus V_n)$ using the natural isomorphisms $\tau^{k_1}_{V_1, V_2 \oplus \dots \oplus V_n}, \tau^{k_2}_{V_2, V_3 \oplus \dots \oplus V_n}, \dots \tau^{k_{n-1}}_{V_{n-1}, V_n}$ as follows
	\begin{align*}
		F_n(V_1 \oplus \dots \oplus V_n) 
		& \cong \bigoplus_{i_1 + \dots + i_n = n}  F_{i_1}(V_1) \otimes \dots \otimes F_{i_n}(V_n)\ .
	\end{align*}
	Naturality of the above decomposition implies that $F_n(M_{\lambda_1, \dots, \lambda_n})$ acts via multiplication by $\lambda_1^{i_1}\cdots \lambda_n^{i_n}$ on the summand $F_{i_1}(V_1) \otimes \dots \otimes F_{i_n}(V_n)$. Therefore the natural isomorphism used to obtain the decomposition identifies the linearization of $F_n$ with the functor $(V_1, \dots, V_n) \mapsto F_1(V_1) \otimes \dots \otimes F_1(V_n)$.
\end{proof}

\begin{remark} \label{rem:constant_term}
	By \cite[p.~150, Rem.~2]{book:Macdonald} we have that for any polynomial functor $F$ with homogeneous components $(F_n)_{n \in \N_0}$ all objects $F_0(V)$ (i.e.\ the ``constant terms'') are canonically isomorphic to $F_0(0)$. For an exponential polynomial functor the unit transformation $\iota$ from Def.~\ref{def:exp_functor} identifies $F_0(0)$ with $F(0) \cong \C$. Thus, the above theorem is still true for $n = 0$ if we define $L_0$ to be the functor from the trivial category on one object to $\Vfin$ mapping the object to $\C$.
\end{remark}

\subsection{The $R$-matrix of an exponential functor} \label{sec:R-matrix}
Consider a polynomial exponential functor $F \colon \Vfin \to \Vfin$ with homogeneous components $(F_n)_{n \in \N_0}$. Note that the permutation group $S_n$ acts unitarily on $V^n = V \oplus \dots \oplus V$ by permuting the summands. Therefore it also acts on $F_n(V \oplus \dots \oplus V)$ by unitary transformations. We will denote this representation by $\eta$. If $M_{\lambda_1, \dots, \lambda_n}$ denotes the multiplication operator defined above and $\sigma \in S_n$, then the following diagram commutes
\[
\begin{tikzcd}[column sep=3cm]
	F_n(V \oplus \dots \oplus V) \ar[r,"F_n(M_{\lambda_1, \dots, \lambda_n})"] \ar[d,"\eta(\sigma)" left] & F_n(V \oplus \dots \oplus V) \ar[d,"\eta(\sigma)"] \\
	F_n(V \oplus \dots \oplus V) \ar[r,"F_n(M_{\lambda_{\sigma(1)}, \dots, \lambda_{\sigma(n)}})" below] & F_n(V \oplus \dots \oplus V)
\end{tikzcd} 
\]
In particular, the linear map $\eta(\sigma)$ restricts to a unitary automorphism of the linearization $L_{F_n}(V, \dots, V)$. Using Lemma~\ref{lem:linearization} we obtain a unitary representation $\rho^{(n)}$ of $S_n$ on $L_n(\C, \dots, \C)$. Let $W = F_1(\C)$ and consider the unitary automorphism $R \colon W \otimes W \to W \otimes W$ representing the nontrivial element $\tau \in S_2$ on $L_2(\C,\C)$. Denote by $R_i \colon W^{\otimes k} \to W^{\otimes k}$ for $1 \leq i \leq k-1$ the linear map 
\[
	R_i = \id{W^{\otimes (i -1)}} \otimes R \otimes \id{W^{\otimes (k-i-1)}}
\]
i.e.\ it coincides with $R$ on the $i$th and $(i+1)$st tensor factor of $W^{\otimes k}$.

\begin{lemma} \label{lem:R_rep_of_Sn}
	Let $n \in \N$ and let $\tau_i \in S_n$ for $i \in \{1, \dots, n-1\}$ be the transposition of the $i$th and $(i+1)$st element. Let $W = F_1(\C)$ and let $\rho^{(n)} \colon S_n \to U(W^{\otimes n})$ be the representation defined above. Then 
	\[
		\rho^{(n)}(\tau_i) = R_i\ .
	\]
	In particular, the linear map $R$ satisfies $R^2 = \id{W \otimes W}$ and the Yang-Baxter equation 
	\[
		R_1R_2R_1 = R_2R_1R_2\ .
	\]
\end{lemma}

\begin{proof}
	Let $\eta_n$ be the representation of $S_n$ on $F_n(\C \oplus \dots \oplus \C)$ as defined above. The transposition $\tau_i$ interchanges the $i$th and the $(i+1)$st summand and acts as the identity on all other summands. Since the projection $F_n(\C, \dots, \C) \to L_n(\C, \dots, \C)$ and inclusion $L_n(\C, \dots, \C) \to F_n(\C, \dots, \C)$ are natural transformations, it follows that $\rho^{(n)}(\tau_i)$ acts trivially on all tensor factors except for the $i$th and $(i+1)$st. To describe the action on these factors observe that 
		\[ 
			F_1(\C)^{\otimes (i-1)} \otimes F_2(\C \oplus \C) \otimes F_1(\C)^{\otimes (n-i-1)}
		\]
		is isomorphic to a natural direct summand of $F_n(\C \oplus \dots \oplus \C)$ using the same method applied in the proof of Lemma~\ref{lem:linearization}. Restricted to this summand $\eta_n(\tau_i)$ agrees with $\id{F_1(\C)^{\otimes (i-1)}} \otimes \eta_2(\tau) \otimes \id{F_1(\C)^{\otimes (n-i-1)}}$, where $\tau \in S_2$ denotes the nontrivial element. Thus, if we restrict further to $F_1(\C)^{\otimes n}$ we see that $\rho^{(n)}(\tau_i)$ agrees with $R_i$. This also implies that the stated relations hold, since they are satisfied by the generators $\tau_i \in S_n$.  	
\end{proof}

The last lemma shows that polynomial exponential functors naturally give rise to unitary involutive $R$-matrices. These were studied in \cite{preprint:LechnerPennigWood} up to the first of the following two equivalence relations:
\begin{definition} \label{def:equiv_R_matrix}
Two unitary involutive $R$-matrices $R \colon W_R^{\otimes 2} \to W_R^{\otimes 2}$ and $S\colon W_S^{\otimes 2} \to W_S^{\otimes 2}$ are called \emph{(unitarily) equivalent} if the associated representations $\rho^{(n)}_R \colon S_n \to U(W_R^{\otimes n})$ and $\rho^{(n)}_S \colon S_n \to U(W_S^{\otimes n})$ are unitarily equivalent for all $n \in \N$. We will denote this by $R \sim_u S$.

Two $R$-matrices as above will be called \emph{strongly equivalent} if they are equivalent in the sense of the last definition and the unitary intertwiner $v_n \colon W_R^{\otimes n} \to W_S^{\otimes n}$ between $\rho_R^{(n)}$ and $\rho_S^{(n)}$ can be chosen to be of the form 
	\[
		v_n = v_1 \otimes \dots \otimes v_1
	\]
	for a unitary isomorphism $v_1 \colon W_R \to W_S$. Since this relation will be seen to be closely related to monoidal equivalence of exponential functors, we will denote it by $R \moneq S$. 
\end{definition}

The following was proven in \cite[Appendix~A, Thm.~5.3]{book:Macdonald} and provides a complete classification of homogeneous polynomial functors in terms of representations of permutation groups.
\begin{theorem} \label{thm:poly_func_classification}
	Let $F_n$ be a homogeneous polynomial functor of degree $n$ with linearization $L_{F_n}$. Then there exists a natural isomorphism of functors 
	\[
		F_n(V) \cong (L_{F_n}(\C, \dots, \C) \otimes V^{\otimes n})^{S_n} \ ,
	\]
	where $S_n$ acts on $L_{F_n}(\C, \dots, \C)$ via the representation induced by permuting the summands of $\C \oplus \dots \oplus \C$ as described above and on $V^{\otimes n}$ by permuting the tensor factors.
\end{theorem}
We will use the above theorem to identify isomorphism classes of polynomial exponential functors with equivalence classes of $R$-matrices. The first step is to show the compatibility of the various equivalence relations. We will start by comparing $\funeq$ and $\sim_u$ in the next lemma and defer the comparison of the monoidal equivalences to the end of the next section.

\begin{lemma} \label{lem:F_R_equivalences}
Let $F,G \colon \Vfin \to \Vfin$ be polynomial exponential functors. Let $R,S$ be the $R$-matrices associated to $F$ and $G$ respectively. Then 
\begin{align*}
	F \funeq G \quad & \Leftrightarrow \quad R \sim_u S \ .
\end{align*}
\end{lemma}

\begin{proof}
	First suppose that $F \funeq G$ and let $v(X) \colon F(X) \to G(X)$ be a  unitary isomorphism between $F$ and $G$. Let $u^F_n(X), u^G_n(X)$ be the orthogonal projections appearing as coefficients of $\lambda^n$ in $F(\lambda\,\id{X})$, $G(\lambda\,\id{X})$, respectively. By naturality the following diagram commutes
	\[
		\begin{tikzcd}
			F(X) \ar[r,"v(X)"] \ar[d,"F(\lambda\,\id{X})" left] & G(X) \ar[d,"G(\lambda\,\id{X})"]\\
			F(X) \ar[r,"v(X)" below] & G(X)
		\end{tikzcd}
	\] 
	and a comparison of coefficients implies that $v(X)\,u^F_n(X) = u^G_n(X)\,v(X)$. Thus, $v(X)$ restricts to natural unitary isomorphisms 
	\[
		v_n(X) \colon F_n(X) \to G_n(X)
	\]
	of the homogeneous summands of degree $n$. Now fix $n \in \N$ and consider the operator $M_{\lambda_1,\dots, \lambda_n} \colon V \oplus \dots \oplus V \to V \oplus \dots \oplus V$ as above. Naturality implies that the following diagram commutes as well:
	\[
		\begin{tikzcd}[column sep=2.4cm]
			F_n(V \oplus \dots \oplus V) \ar[r,"v_n(V\oplus \dots \oplus V)"] \ar[d,"F_n(M_{\lambda_1, \dots, \lambda_n})" left] & G_n(V \oplus \dots \oplus V) \ar[d,"G_n(M_{\lambda_1, \dots, \lambda_n})"]\\
			F_n(V \oplus \dots \oplus V) \ar[r,"v_n(V\oplus \dots \oplus V)" below] & G_n(V \oplus \dots \oplus V)
		\end{tikzcd}
	\]
	and in particular $v_n(\C \oplus \dots \oplus \C)$ restricts to a unitary isomorphism $v^{L}_n$ of the linearizations of $F$ and $G$, i.e.
	\[
		v^{L}_n \colon L_{F_n}(\C, \dots, \C) \to L_{G_n}(\C, \dots, \C)\ .
	\]
	Let $W_R = F_1(\C)$ and $W_S = G_1(\C)$. By Lemma~\ref{lem:linearization} the linearizations can be identified with $L_n^F(\C, \dots, \C) = W_R^{\otimes n}$ and $L_n^G(\C, \dots, \C) = W_S^{\otimes n}$, respectively. Let $v^{R,S}_n \colon W_R^{\otimes n} \to W_S^{\otimes n}$ be the unitary corresponding to $v^L_n$ under this identification.	A similar argument as above shows that $v^{R,S}_n$ intertwines the two representations $\rho_R^{(n)}\colon S_n \to U(W_R^{\otimes n})$ and $\rho_S^{(n)}\colon S_n \to U(W_S^{\otimes n})$ for each $n \in \N$, since these are induced by the isomorphism permuting the summands of $V \oplus \dots \oplus V$. Hence, $R \sim_{u} S$.
	
	To prove the other direction let $W_R = F_1(\C)$, $W_S = G_1(\C)$ and assume that $R \sim_u S$. Let $v_n \colon W_R^{\otimes n} \to W_S^{\otimes n}$ be the isometric intertwiner between the representations $\rho_R^{(n)}$ and $\rho_S^{(n)}$. By Thm.~\ref{thm:poly_func_classification} there are natural isomorphisms
	\begin{align}
		F_n(V) \cong (W_R^{\otimes n}\otimes V^{\otimes n})^{S_n} \ , \notag \\
		G_n(V) \cong (W_S^{\otimes n}\otimes V^{\otimes n})^{S_n} \ , \label{eqn:lin-iso}
	\end{align}
	where $S_n$ acts on $W_R^{\otimes n}$, $W_S^{\otimes n}$ via $\rho_R^{(n)}$, $\rho_S^{(n)}$, respectively. Therefore $v_n$ induces a natural unitary isomorphism 
	\(
		v_n(V) \colon F_n(V) \to G_n(V)
	\). 
	Since $F(V) \cong \bigoplus_{n \in \N_0} F_n(V)$ and similarly for $G$, this isomorphism extends to a natural unitary isomorphism between $F$ and $G$. Thus, $F \funeq G$, which finishes the proof of the statement.
\end{proof}

\subsubsection{Thoma parameters} \label{sec:Thoma_par}
Denote by $S_{\infty}$ the infinite symmetric group, i.e.\ the union over all $S_n$ with respect to the inclusions $S_n \to S_{n+1}$ induced by permuting the first $n$ elements of $\{1, \dots, n+1\}$. The elements of $S_{\infty}$ are precisely all permutations of $\N$ with finite support. Let $\tau_{n} \colon M_d(\C)^{\otimes n} \to \C$ be the normalised trace. Let $\chi_R^{(n)} \colon S_n \to \C$ be defined by $\chi_R^{(n)} = \tau_n \circ \rho_R^{(n)}$. Then $\chi_R^{(n+1)}$ restricts to $\chi_R^{(n)}$ on $S_n$. Thus, the sequence $(\chi_R^{(n)})_{n \in \N}$ gives rise to a character $\chi_R \colon S_{\infty} \to \C$. It turns out that $\chi_R$ belongs to the class of extremal characters on $S_{\infty}$. These were classified by Thoma \cite{paper:Thoma}: Let $\mathbb{T}$ denote the collection of all sequences $\{\alpha_i\}_{i \in \N}$, $\{\beta_i\}_{i \in \N}$ of real numbers with the following properties
\begin{enumerate}[i)]
	\item $\alpha_i \geq 0$ and $\beta_i \geq 0$,
	\item $\alpha_i \geq \alpha_{i+1}$ and $\beta_i \geq \beta_{i+1}$,
	\item $\sum_i \alpha_i + \sum_j \beta_j \leq 1$.
\end{enumerate}
Then the extremal characters $\chi$ are in $1:1$-correspondence with the elements of $\mathbb{T}$. On an $n$-cycle $c_n \in S_{\infty}$ for $n \geq 2$ the character $\chi$ corresponding to $((\alpha_i)_{i \in \N}, (\beta_i)_{i \in \N})$ takes the value
\[
	\chi(c_n) = \sum_{i} \alpha_i^n + (-1)^{n+1} \sum_{i} \beta_i^n\ .
\]
We will call the pair $((\alpha_i)_{i \in \N}, (\beta_i)_{i \in \N})$ the \emph{Thoma parameters} of $\chi$. The two characters $\chi_R$ and $\chi_S$ associated to the two $R$-matrices $R$ and $S$ agree if and only if $\dim(W_R) = \dim(W_S)$ and $R \sim_u S$ (see Def.~\ref{def:equiv_R_matrix}). 
\pagebreak[10]

The main result of~\cite{preprint:LechnerPennigWood} states that the characters $\chi_R$ are parametrised by Thoma parameters that lie in the following subspace
\begin{enumerate}[i)]
	\item only finitely many $\alpha_i$, $\beta_i$ are non-zero,
	\item all $\alpha_i$ and $\beta_i$ are rational,
	\item $\sum_i \alpha_i + \sum_j \beta_j = 1$.
\end{enumerate}
Since we consider exponential functors $F \colon \Vfin \to \Vfin$, finite-dimensionality imposes a constraint that has consequences for the possible Thoma parameters of $\chi_R$ for the $R$-matrix associated to $F$.

\begin{lemma} \label{lem:Thoma_parameters}
Let $F \colon \Vfin \to \Vfin$ be a polynomial exponential functor, let $W = F_1(\C)$ and let $R \colon W \otimes W \to W \otimes W$ be the $R$-matrix associated to $F$. Then the Thoma parameters of the extremal character $\chi_R$ associated to $R$ satisfy the following conditions:
\begin{enumerate}[a)]
	\item all $\alpha_i = 0$,
	\item only finitely many $\beta_j$ are non-zero, rational and add up to $1$.
\end{enumerate}  
\end{lemma}

\begin{proof}
Part b) of the statement is clear from the classification of extremal characters arising from unitary involutive $R$-matrices. Thus, it remains to prove part a). Since $F(\C)$ is finite-dimensional, only finitely many summands in its homogeneous decomposition are nonzero. Let $N \in \N$ have the property that $F_k(\C) = 0$ for all $k \geq N$. By Thm.~\ref{thm:poly_func_classification} the dimension of the space $F_m(\C)$ agrees with the multiplicity of the trivial representation $\iota_{\C} \colon S_m \to U(1)$ in $\rho^{(m)}_R$. Denote by $\langle \lambda, \rho \rangle$ the multiplicity of the irreducible representation $\lambda$ of $S_n$ in the representation $\rho$. By the considerations above we conclude that 
	\[
		\langle \iota_{\C}, \rho_R^{(m)} \rangle  = 0
	\]
	for all $m \geq N$. It was shown in \cite[Prop.~5.7]{preprint:LechnerPennigWood} that this condition holds if and only if the Young diagram of $\iota$ contains a rectangle of height $\ell(\alpha) + 1$ and width $\ell(\beta) +1$, where $\ell(\alpha)$ and $\ell(\beta)$ are the numbers of nonzero $\alpha_i$ and $\beta_i$, respectively. The Young diagram of the trivial representation $\iota$ of $S_m$ is a rectangle of height $1$ and width $m$. In particular, we must have $\ell(\alpha) = 0$.
\end{proof}

The Thoma parameters do not take into account the dimension $d_R = \dim(W)$ of the vector space the $R$-matrix acts on. Thus, it is often more convenient to work with the \emph{rescaled Thoma parameters} 
\[
	a_i = d_R\alpha_i \qquad, \qquad b_i = d_R\beta_i\ .
\]
By \cite[Thm.~3.6]{preprint:LechnerPennigWood} these are natural numbers with the property that $R \sim_u S$ if and only if the rescaled Thoma parameters of $R$ and $S$ agree \cite[Thm.~4.8]{preprint:LechnerPennigWood}.

\subsection{Classification of exponential functors}
In the following we will give examples of polynomial exponential functors $F^W$ (depending on a vector space $W$), which have rescaled Thoma parameters $(0,b_1)$ for any $b_1 \in \N$. Afterwards we extend this to exhaust the list of all possible Thoma parameters discovered in Lemma~\ref{lem:Thoma_parameters}.

Observe that the exterior algebra functor $V \mapsto \extp^*(V)$ is polynomial and exponential. We will examine a modified version of it defined as follows: Fix a finite-dimensional inner product space $W$ and consider the functor 
\begin{equation} \label{eqn:mod_ext_power}
	F^{W}(V) = \bigoplus_{k=0}^{\infty}\,W^{\otimes k} \otimes \extp^k(V)\ ,
\end{equation}
where we define $W^{\otimes 0} = \C$. Since $\extp^k(V) = 0$ for $k > \dim(V)$, the sum is actually finite and the functor $F^{W} \colon \Vfin \to \Vfin$ is well-defined. It is also a polynomial functor with homogeneous components
\[
	F^W_k(V) = W^k \otimes \extp^k(V)\ .
\]
since $V \mapsto \extp^k(V)$ is homogeneous of degree $k$. Observe that there is a natural equivalence  
\begin{align*}
	F^W_k(V_1 \oplus V_2) &=\ W^{\otimes k} \otimes \extp^k(V_1 \oplus V_2) \\
	& \cong\ \bigoplus_{i+j=k} \,\left(W^{\otimes i} \otimes \extp^iV_1\right) \otimes \left(W^{\otimes j} \otimes \extp^jV_2\right) \\
	& = \bigoplus_{i+j=k} F^W_i(V_1) \otimes F^W_j(V_2)\ ,
\end{align*}
where the second line follows from the first by applying the isomorphism $\extp^k(V_1 \oplus V_2) \cong \bigoplus_{i+j=k} \extp^i V_1 \otimes \extp^j V_2$ and then interchanging $W^{\otimes j}$ with $\extp^i V_1$. Moreover, there is a natural isomorphism $F^W(0) \cong \C$. We can extend these identifications to natural unitary isomorphisms
\[
	\tau_{V_1,V_2} \colon F^W(V_1 \oplus V_2) \to F^W(V_1) \otimes F^W(V_2)
\]
and $\iota \colon F^W(0) \to \C$, which satisfy the associativity and unitality conditions in Def.~\ref{def:exp_functor} and therefore turn $F^W$ into a polynomial exponential functor. The linearization of $F^W$ can now be read off from Lemma~\ref{lem:linearization} and is naturally equivalent to 
\begin{align*}
	L_n(V_1, \dots, V_n) &\cong F^W_1(V_1) \otimes \dots \otimes F^W_1(V_n) \cong W \otimes V_1 \otimes \dots \otimes W \otimes V_n \\
	& \cong W^{\otimes n} \otimes V_1 \otimes \dots \otimes V_n\ .
\end{align*}
To understand the $R$-matrix $R \colon W^{\otimes 2} \to W^{\otimes 2}$ we have to look at the $S_2 = \Z/2\Z$-action on the space
\[
	F^W_2(\C \oplus \C) = W^{\otimes 2} \otimes \extp^2(\C \oplus \C)
\]
induced by interchanging the two summands of $\C\oplus \C$. On the top degree summand of $\extp^2(\C \oplus \C)$ this transposition corresponds to multiplication by~$(-1)$. Thus, the $R$-matrix is given by $-\id{W \otimes W}$. Using \cite[Thm.~4.8]{preprint:LechnerPennigWood} we compute the Thoma parameters of the $R$-matrix of $F^W$ to be $(0,\dim(W))$. We summarise these results in the next lemma:

\begin{lemma} \label{lem:exp_functor_FW}
	The functor $F^{W}$ defined in (\ref{eqn:mod_ext_power}) is polynomial and exponential with respect to the unitary natural isomorphisms $\tau$ and $\iota$ defined above. The linearization of $F^W$ is naturally equivalent to $L_n^W$ with 
	\[
		L_n^W(V_1, \dots, V_n) = W^{\otimes n} \otimes V_1 \otimes \dots \otimes V_n
	\]
	Let $d_W = \dim(W)$, then the rescaled Thoma parameters of the extremal character $\chi_R$ associated to the $R$-matrix of $F^{W}$ are $(0,d_W)$. 
\end{lemma}

Given two $R$-matrices $R \colon W_R^{\otimes 2} \to W_R^{\otimes 2}$ and $S \colon W_S^{\otimes 2} \to W_S^{\otimes 2}$ we can define another $R$-matrix denoted by $R \boxplus S \colon (W_R \oplus W_S)^{\otimes 2} \to (W_R \oplus W_S)^{\otimes 2}$ as follows: On the summand $W_R^{\otimes 2}$ it agrees with $R$, on $W_S^{\otimes 2}$ it agrees with $S$ and on $(W_R \otimes W_S) \oplus (W_S \otimes W_R)$ it is given by interchanging the two summands. This is \cite[Def.~4.1]{preprint:LechnerPennigWood} and the details of this construction can be found in \cite[Sec.~4.1]{preprint:LechnerPennigWood}. We will see that this operation corresponds to forming tensor products of exponential functors. 

Let $F$ and $G$ be polynomial exponential functors. Observe that $V \mapsto (F \otimes G)(V) = F(V) \otimes G(V)$ is again polynomial. It is also exponential with respect to the natural transformation $\tau^{F \otimes G}$ induced by 
\[
\tau^F \otimes \tau^G \colon F(V \oplus W) \otimes G(V \oplus W) \to F(V) \otimes F(W) \otimes G(V) \otimes G(W)
\]
and the isomorphism interchanging the two middle tensor factors. The unit transformation $\iota^{F \otimes G} \colon F(0) \otimes G(0) \to \C$ is given by $\iota^F \otimes \iota^G$ followed by the canonical isomorphism $\C \otimes \C \cong \C$. 

Let $F(\lambda\,\id{V}) = \sum_{n \in \N_0} u^F_n(V)\lambda^n$ and $G(\lambda\,\id{V}) = \sum_{n \in \N_0} u^G_n(V)\lambda^n$ be the decompositions of the multiplication operator. The coefficients are the projections onto the homogeneous components of $F$ and $G$, respectively. Then we have 
\[
	(F \otimes G)(\lambda\,\id{V}) = F(\lambda\,\id{V}) \otimes G(\lambda\,\id{V}) = \sum_{i,j} u^F_i(V) \otimes u^G_j(V) \lambda^{i+j}\ .
\]
This implies that the degree $n$ homogeneous component $(F \otimes G)_n$ of $F \otimes G$ is given by 
\[
	(F \otimes G)_n(V) = \bigoplus_{i+j = n} F_i(V) \otimes G_j(V)\ .
\]
In particular, we obtain $(F \otimes G)_1(V) \cong F_0(V) \otimes G_1(V) \oplus F_1(V) \otimes G_0(V)$. Thus, combining Lemma~\ref{lem:linearization} and Remark~\ref{rem:constant_term} we see that the linearization of $F\otimes G$ is naturally isomorphic to 
\[
	L^{F \otimes G}_n(V_1, \dots, V_n) = (F_1(V_1) \oplus G_1(V_1)) \otimes \dots \otimes (F_1(V_n) \oplus G_1(V_n))\ . 
\]
To identify the permutation part of $R \boxplus S$ we need the following Lemma.

\begin{lemma} \label{lem:action_on_F1}
	Let $F \colon \Vfin \to \Vfin$ be a polynomial exponential functor with homogeneous decomposition $(F_n)_{n \in \N_0}$. Let $\tau_{\C,\C}^1 \colon F_1(\C \oplus \C) \to F_1(\C) \oplus F_1(\C)$ be the natural isomorphism induced by $\tau$ as in Lemma~\ref{lem:exp-homogeneous}, where we identify $F_0(\C) \cong \C$ (see Rem.~\ref{rem:constant_term}). Let $f_V \colon V\oplus V \to V \oplus V$ be the map interchanging the two summands. Then the following diagram commutes
	\[
		\begin{tikzcd}[column sep=1.7cm]
			F_1(\C^2) \ar[r,"F_1(f_{\C})"] \ar[d,"\tau^1_{\C,\C}" left] & F_1(\C^2) \ar[d,"\tau^1_{\C,\C}"] \\
			F_1(\C) \oplus F_1(\C) \ar[r,"f_{F_1(\C)}" below] & F_1(\C) \oplus F_1(\C)
		\end{tikzcd}  
	\]
\end{lemma}

\begin{proof}
First note $F_1(0) = 0$ and the map $F_1(\C \oplus 0) \to F_1(\C)$ induced by $\tau^1_{\C,0} \colon F_1(\C \oplus 0) \to F_1(\C) \oplus 0 \cong F_1(\C)$ coincides with the one induced by the canonical isomorphism $\C \oplus 0 \to \C$ as a consequence of the unitality condition in Def.~\ref{def:exp_functor} c). Let $p^V_i \colon V \oplus V \to V$ for $i \in \{1,2\}$ be the projection onto the $i$th summand. Thus, by naturality the following diagram commutes
	\[
		\begin{tikzcd}[column sep=2cm]
			F_1(\C \oplus \C) \ar[r,"F_1(\id{\C} \oplus 0)"] \ar[d,"\tau^1_{\C,\C}" left] & F_1(\C \oplus 0) \ar[d] \\
			F_1(\C) \oplus F_1(\C) \ar[r,"p^{F_1(\C)}_1" below] & F_1(\C)
		\end{tikzcd}
	\]
	where the unlabelled arrow on the right is induced by $\C \oplus 0 \to \C$. In particular, we obtain 
	\(
		p_1^{F_1(\C)} \circ \tau^1_{\C,\C} = F_1(p_1^{\C})\ 
	\)
	and similarly with $p_1$ exchanged by $p_2$. Let $\iota^V_j \colon V \to V \oplus V$ be the inclusion onto the $j$th summand. A similar argument shows that 
	\[
		\tau^1_{\C,\C} \circ F_1(\iota_j^{\C}) = \iota_j^{F_1(\C)}
	\]
	Since $F_1$ is homogeneous of degree $1$, it is additive. Note that $f_V$ can be expressed as follows: $f_V = \iota_1^V \circ p_2^V + \iota_2^V \circ p_1^V$. Thus,
	\begin{align*}
		   & \tau^1_{\C,\C} \circ F_1(f_{\C})
		= \tau_{\C,\C}^1 \circ F_1(\iota_1^{\C} \circ p_2^{\C} + \iota_2^{\C} \circ p_1^{\C}) \\
		=\ & \tau_{\C,\C}^1 \circ F_1(\iota_1^{\C}) \circ F_1(p_2^{\C}) + \tau_{\C,\C}^1 \circ F_1(\iota_2^{\C}) \circ F_1(p_1^{\C}) \\
		=\ & \iota_1^{F_1(\C)} \circ p_2^{F_1(\C)} \circ \tau^1_{\C,\C} + \iota_2^{F_1(\C)} \circ p_1^{F_1(\C)} \circ \tau^1_{\C,\C} \\
		=\ & f_{F_1(\C)} \circ \tau^1_{\C,\C} \qedhere
	\end{align*}
\end{proof}

Now we are able to prove our previous guess about the $R$-matrix of tensor products:
\begin{theorem} \label{thm:box_sum}
	Let $F, G$ be polynomial exponential functors with homogeneous decompositions $(F_n)_{n \in \N_0}$ and $(G_n)_{n \in \N_0}$, respectively. Let $W_R = F_1(\C)$ and $W_S = G_1(\C)$ and denote by $R \colon W_R^{\otimes 2} \to W_R^{\otimes 2}$ and $S \colon W_S^{\otimes 2} \to W_S^{\otimes 2}$ the $R$-matrices associated to $F$ and $G$, respectively. Then the $R$-matrix associated to $F \otimes G$ is $R \boxplus S$.
\end{theorem}

\begin{proof}
	Let $\tau \in S_2 \cong \Z/2\Z$ be the non-trivial element, which we will identify with the isomorphism interchanging the two summands of $\C^2$. The $R$-matrix of $F \otimes G$ is the restriction of $(F\otimes G)_2(\tau)$ on $(F\otimes G)_2(\C^2)$ to  
	\[
		L = (F \otimes G)_1(\C) \otimes (F \otimes G)_1(\C) \subset (F\otimes G)_2(\C^2)\ .
	\]
	By our observation about tensor products of polynomial functors the 2-homogeneous summand $(F \otimes G)_2(\C^2)$ is naturally isomorphic to
	\begin{equation} \label{eqn:sum_decomposition}
		F_2(\C^2) \otimes G_0(\C^2) \quad \oplus \quad F_1(\C^2) \otimes G_1(\C^2) \quad \oplus \quad F_0(\C^2) \otimes G_2(\C^2)\ .
	\end{equation}
	Observe that $(F \otimes G)_2(\tau)$ acts via $F_i(\tau) \otimes G_j(\tau)$ on $F_i(\C^2) \otimes G_j(\C^2)$ in the above decomposition. The space $L$ is obtained from (\ref{eqn:sum_decomposition}) in the following way: The first summand of (\ref{eqn:sum_decomposition}) contains $F_1(\C) \otimes F_1(\C) \otimes G_0(\C) \otimes G_0(\C)$, on which $(F \otimes G)_2(\tau)$ acts like $R$ after identifying $G_0(\C)$ with $\C$. The last summand contains $F_0(\C) \otimes F_0(\C) \otimes G_1(\C) \otimes G_1(\C)$, on which $(F \otimes G)_2(\tau)$ restricts to $S$ after identifying $F_0(\C) \cong \C$. The second summand contains
	\begin{equation} \label{eqn:middle_summand}
		F_1(\C) \otimes F_0(\C) \otimes G_0(\C) \otimes G_1(\C) \ \oplus \ F_0(\C) \otimes F_1(\C) \otimes G_1(\C) \otimes G_0(\C)\ .
	\end{equation}
	Since $(F \otimes G)_2(\tau)$ acts like $F_1(\tau) \otimes G_1(\tau)$ on the second summand of (\ref{eqn:sum_decomposition}), Lemma~\ref{lem:action_on_F1} implies that $(F \otimes G)_2(\tau)$ restricts to the action that interchanges the two summands of (\ref{eqn:middle_summand}) after identifying $F_0(\C)$ and $G_0(\C)$ with $\C$. These are all summands of $L$. Thus, we see that the $R$-matrix of $F \otimes G$ turns out to be $R \boxplus S$.
\end{proof}

Let $R$, $S$ be two unitary involutive $R$-matrices with rescaled Thoma parameter sets $(a, b) = (\{a_1, \dots, a_m\}, \{b_1, \dots, b_n\})$ and $(a', b')$, respectively. Then the parameter set of $R \boxplus S$ is given by $(a \cup a', b \cup b')$ \cite[Prop.~4.4 ii) and (4.14)]{preprint:LechnerPennigWood}.
\begin{theorem} \label{thm:classification_part_1}
	For any polynomial exponential functor $F \colon \Vfin \to \Vfin$ there is a sequence of finite-dim.\ inner product spaces $W_1, \dots, W_n \in \obj{\Vfin}$ with the property that 
	\[
		F \funeq F^{W_1} \otimes \dots \otimes F^{W_n}\ .
	\]
\end{theorem}

\begin{proof}
	Let $(F_n)_{n \in \N_0}$ be the homogeneous components of $F$, let $W = F_1(\C)$, $d_W = \dim(W)$ and let $R \colon W^{\otimes 2} \to W^{\otimes 2}$ be the $R$-matrix associated to~$F$.	By Lemma~\ref{lem:Thoma_parameters} the rescaled Thoma parameters of $R$ have the form 
	\[
		(a,b) = (0, (b_1, \dots, b_n))
	\] 
	for positive integers $b_i$ that satisfy $\sum_i b_i = d_W$. Let $W_i = \C^{b_i}$. Let $R_i$ be the $R$-matrix of the polynomial exponential functor $F^{W_i}$. Its rescaled parameters are $(0,b_i)$ by Lemma~\ref{lem:exp_functor_FW}. By Thm.~\ref{thm:box_sum} the $R$-matrix of $F^{W_1} \otimes \dots \otimes F^{W_n}$ is $R_1 \boxplus \dots \boxplus R_n$, the parameters of which are $(0,(b_1, \dots, b_n))$ by the observation in the preceding paragraph. In particular, the $R$-matrices of $F$ and $F^{W_1} \otimes \dots \otimes F^{W_n}$ are equivalent in the sense of Def.~\ref{def:equiv_R_matrix} and the statement follows from Lemma~\ref{lem:F_R_equivalences}.
\end{proof}

\subsection{From $R$-matrices to exponential functors}
Theorem~\ref{thm:classification_part_1} gives a full classification of polynomial exponential functors up to the first equivalence relation $\funeq$. In this section we study how monoidal equivalence of exponential functors is related to strong equivalence of their respective $R$-matrices. We start by constructing a polynomial exponential functor from a given $R$-matrix with the right Thoma parameters. For a given $V \in \obj{\Vfin}$ let $T \colon V^{\otimes 2} \to V^{\otimes 2}$ be the isomorphism interchanging the tensor factors. Let $R \boxtimes T \colon (W \otimes V)^{\otimes 2} \to (W \otimes V)^{\otimes 2}$ be the tensor product of the two $R$-matrices \cite[Eq.~(4.15)]{preprint:LechnerPennigWood}. Define
\begin{equation} \label{eqn:def_of_FR}
	F^R(V) = \bigoplus_{n \in \N_0} ((W \otimes V)^{\otimes n}))^{S_n}
\end{equation}
where $S_n$ acts on $(W \otimes V)^{\otimes n}$ via $\rho^{(n)}_{R \boxtimes T}$, $S_0$ and $S_1$ are trivial groups and we set $(W \otimes V)^{\otimes 0} = \C$. A priori this is a polynomial functor from $\Vfin$ to the category $\mathcal{V}_{\C}$ of (not necessarily finite-dimensional) inner product spaces. Its homogeneous components $(F^R_n)_{n \in \N_0}$ are given by 
\[
	F^R_n(V) = ((W \otimes V)^{\otimes n})^{S_n}\ .
\]
To see that $F^R$ is in fact an exponential functor, we have to analyse its behaviour with respect to direct sums. Let $i,j \in \N_0$ with $i + j = n$ and consider the embedding
\[
	\iota_{i,j} \colon (W \otimes V_1)^{\otimes i} \otimes (W \otimes V_2)^{\otimes j} \to (W \otimes (V_1 \oplus V_2))^{\otimes n}
\]
induced by the two inclusions $W \otimes V_k \to W \otimes (V_1 \oplus V_2)$ for $k \in \{1,2\}$. Let $T_k$ be the $R$-matrix interchanging the tensor factors of $V_k^{\otimes 2}$. The analogous $R$-matrix for $V_1 \oplus V_2$ is $T_1 \boxplus T_2$. In particular, 
\[
	F^R_n(V_1 \oplus V_2) = ((W \otimes (V_1 \oplus V_2))^{\otimes n})^{S_n}\ ,
\]
where $S_n$ acts via $\rho_{R \boxtimes (T_1 \boxplus T_2)}^{(n)}$. Let $X_{i,j} = S_{n}/(S_i \times S_j)$, $N_{i,j} = \binom{n}{i}$ and choose representatives $\sigma_1 = e, \sigma_2, \dots, \sigma_{N_{i,j}} \in S_n$ for each element of $X_{i,j}$. Define the homomorphism 
\[
	\varphi_{i,j} \colon (W \otimes V_1)^{\otimes i} \otimes (W \otimes V_2)^{\otimes j} \to (W \otimes (V_1 \oplus V_2))^{\otimes n}
\]
by $\iota_{i,j}$ followed by averaging over the action of the representatives of $X_{i,j}$, i.e.\ by
\begin{equation} \label{eqn:phi_ij}
	\varphi_{i,j}(x) = \frac{1}{\sqrt{N_{i,j}}} \sum_{k=1}^{N_{i,j}} \sigma_k \cdot \iota_{i,j}(x)\ ,
\end{equation}
where $S_n$ again acts by $\rho^{(n)}_{R \boxtimes (T_1 \boxplus T_2)}$. Note that the definition of $\varphi_{i,j}$ depends on the choice of representatives $\sigma_1, \dots, \sigma_{N_{i,j}}$.

To understand the action of $S_n$ we employ the following description of $(V_1 \oplus V_2)^{\otimes n}$: Fix $i,j \in \N_0$ with $i + j = n$. Let
\[
	a_{i,j} \colon \{1, \dots, n\} \to \{1,2\} \qquad , \qquad x \mapsto 
	\begin{cases}
		1 & \text{for } 1 \leq x \leq i \\
		2 & \text{else}
	\end{cases}
\]
Note that $a_{i,j} \circ \tau = a_{i,j}$ for all $\tau \in S_i \times S_j$. We have an isomorphism
\begin{equation} \label{eqn:tensor_of_sum}
	(V_1 \oplus V_2)^{\otimes n} \cong \bigoplus_{i+j = n} \bigoplus_{[\sigma] \in X_{i,j}} V_{i,j, [\sigma]}
\end{equation}
with $V_{i,j,[\sigma]} = V_{a_{i,j}(\sigma(1))} \otimes \dots \otimes V_{a_{i,j}(\sigma(n))}$. An element $\tau \in S_n$ maps $V_{i,j,[\sigma]}$ to $V_{i,j,[\tau \cdot \sigma]}$. Thus, our choice of representatives implies that all summands in $\varphi_{i,j}(x)$ are orthogonal. In particular, $\varphi_{i,j}$ is injective. The factor $\tfrac{1}{\sqrt{N_{i,j}}}$ ensures that $\varphi_{i,j}$ also preserves inner products. Define
\[
	\widehat{\kappa}^n_{V_1,V_2} \colon \bigoplus_{i+j = n} (W \otimes V_1)^{\otimes i} \otimes (W \otimes V_2)^{\otimes j} \to (W \otimes (V_1 \oplus V_2))^{\otimes n}
\] 
as the sum over all $\varphi_{i,j}$. The decomposition (\ref{eqn:tensor_of_sum}) shows that each each summand of the domain is mapped to a different orthogonal summand of the codomain. Hence, $\widehat{\kappa}^n_{V_1,V_2}$ is still injective.

Given $k \in \{1,\dots,N_{i,j}\}$ and $\sigma \in S_n$, there is $\ell(k)$ and $\tau \in S_i \times S_j$ with the property that $\sigma \cdot \sigma_k = \sigma_{\ell(k)}\cdot \tau$ and for fixed $\sigma$ the map $k \mapsto \ell(k)$ is a bijection. Let $x \in ((W \otimes V_1)^{\otimes i})^{S_i} \otimes ((W \otimes V_2)^{\otimes j})^{S_j}$. Then we obtain
\begin{align*}
	\sigma \cdot \varphi_{i,j}(x) &= \frac{1}{\sqrt{N_{i,j}}} \sum_{k=1}^{N_{i,j}} (\sigma \cdot \sigma_k) \cdot \iota_{i,j}(x) = \frac{1}{\sqrt{N_{i,j}}} \sum_{k=1}^{N_{i,j}} \sigma_{\ell(k)} \cdot \tau \cdot \iota_{i,j}(x) \\
	&= \frac{1}{\sqrt{N_{i,j}}} \sum_{k=1}^{N_{i,j}} \sigma_{k} \cdot \iota_{i,j}(x)	= \varphi_{i,j}(x)
\end{align*}
where we used the equivariance of $\iota_{i,j}$ with respect to the action of $S_i \times S_j$ on both sides. This implies that $\widehat{\kappa}^n_{V_1,V_2}$ restricts to a natural isometry
\[
	\kappa^n_{V_1,V_2} \colon \bigoplus_{i+j = n} ((W \otimes V_1)^{\otimes i})^{S_i} \otimes ((W \otimes V_2)^{\otimes j})^{S_j} \to ((W \otimes (V_1 \oplus V_2))^{\otimes n})^{S_n}
\] 
Note that $\kappa^n_{V_1,V_2}$ no longer depends on our initial choice of coset representatives $\sigma_k$.

\begin{lemma} \label{lem:tau_surjective}
	Let $R \colon W^{\otimes 2} \to W^{\otimes 2}$ be an $R$-matrix with rescaled Thoma parameters $(0,(b_1, \dots, b_m))$. 	Let $F_n^R$ be the functor and $\kappa^n_{V_1,V_2}$ be the natural transformation constructed above. Then $\kappa^n_{V_1,V_2}$ is a unitary isomorphism for every $V_1,V_2 \in \obj{\Vfin}$. In particular, it induces a natural isomorphism 
	\[
		\tau_{V_1, V_2}^n \colon F^R_n(V_1 \oplus V_2) \to \bigoplus_{i+j = n} F^R_i(V_1) \otimes F^R_j(V_2)
	\]
	with $\tau^n_{V_1,V_2} = \left(\kappa^n_{V_1,V_2}\right)^{-1}$. Moreover, $F^R_n(V) = 0$ for sufficiently large $n$ (depending on $V$) and $F^R(V)$ as defined in (\ref{eqn:def_of_FR}) is finite-dimensional.
\end{lemma}

\begin{proof}
	We have already seen that $\kappa^n_{V_1,V_2}$ is injective. Thus, it suffices to show that the dimensions of domain and codomain agree. This is obvious for $n \in \{0,1\}$. Fix $n \geq 2$ and let $U \in \obj{\Vfin}$. Consider the action of $S_n$ on $(W \otimes U)^{\otimes n}$ via $\rho_{R \boxtimes T}^{(n)}$. Let $d_U = \dim(U)$. The $R$-matrix $R \boxtimes T$ has rescaled Thoma parameters $(0,(b_1, \dots, b_1, b_2,\dots, b_2, \dots, b_m, \dots, b_m)$, where each $b_i$ is repeated $d_U$ times \cite[Lemma~4.9]{preprint:LechnerPennigWood}.
	
	The dimension of $((W \otimes U)^{\otimes n})^{S_n}$ agrees with the multiplicity $\langle \iota_{\C}, \rho_{R \boxtimes T}^{(n)}\rangle$ of the trivial representation $\iota_{\C}$ in $\rho^{(n)}_{R \boxtimes T}$. This multiplicity was computed in \cite[Prop.~5.7]{preprint:LechnerPennigWood} in terms of the Thoma parameters and evaluates to
	\begin{align*}
		&\langle \iota_{\C}, \rho_{R \boxtimes T}^{(n)}\rangle = \left[(1 \otimes \omega) \circ \Delta(s_n) \right]((0,\dots, 0), (b_1, \dots, b_1, \dots, b_m, \dots, b_m)) \\
		=\ & e_n(b_1, \dots, b_1, \dots, b_m, \dots, b_m) = \sum_{s_1 + \dots + s_m = n} \binom{d_U}{s_1} \cdots \binom{d_U}{s_m} b_1^{s_1} \cdots b_m^{s_m} 
	\end{align*}
	where we refer the reader to \cite[Sec.~5.2]{preprint:LechnerPennigWood} for the notation used here. For sufficiently large $n$ there will always be at least one summand $s_k$ that is larger than $d_U$. Hence, this expression vanishes for large $n$, which proves the second statement. 
	
	For $k \in \{1,2\}$ let $d_k = \dim(V_k)$ and let $T_k \colon V_k^{\otimes 2} \to V_k^{\otimes 2}$ be the map interchanging the tensor factors. Let $T_{12}$ be the corresponding operation on $(V_1 \oplus V_2)^{\otimes 2}$. Using the following identity for the binomial coefficients  
	\[
		\sum_{i + j = n} \binom{d_1}{i} \binom{d_2}{j} = \binom{d_1 + d_2}{n}
	\] 
	we can compare the dimensions of the domain and codomain of $\kappa^n_{V_1,V_2}$:
	\begin{align*}
		& \sum_{i +j = n} \langle \iota_{\C}, \rho_{R \boxtimes T_1}^{(i)}\rangle\langle \iota_{\C}, \rho_{R \boxtimes T_2}^{(j)}\rangle \\ 
		=\ & \sum_{i +j = n} \sum_{s_1 + \dots + s_m = i \atop t_1 + \dots + t_m = j} \binom{d_1}{s_1} \binom{d_2}{t_1} \cdots \binom{d_1}{s_m} \binom{d_2}{t_m} b_1^{s_1 + t_1} \cdots b_m^{s_m + t_m} \\		
		=\ & \sum_{x_1 + \dots + x_m = n}\!\! \left(\sum_{s_1 + t_1 = x_1}\!\! \binom{d_1}{s_1} \binom{d_2}{t_1}\!\right) \cdots \left(\sum_{s_m + t_m = x_m}\!\! \binom{d_1}{s_m} \binom{d_2}{t_m}\!\right) b_1^{x_1} \cdots b_m^{x_m} \\		
		=\ & \sum_{x_1 + \dots + x_m = n} \binom{d_1 + d_2}{x_1} \cdots \binom{d_1+d_2}{x_m} b_1^{x_1} \cdots b_m^{x_m} 	
		= \langle \iota_{\C}, \rho_{R \boxtimes T_{12}}^{(n)}\rangle		
	\end{align*}
	This proves that $\kappa^n_{V_1,V_2}$ is in fact an isomorphism. Thus, $\tau^n_{V_1,V_2}$ is well-defined. 
\end{proof}

The transformations $\tau_{V_1,V_2}^n$ yield natural isomorphisms 
\[
	\tau_{V_1,V_2} \colon F^R(V_1 \oplus V_2) \to F^R(V_1) \otimes F^R(V_2)\ .
\] 
Moreover, there is a canonical isomorphism $\iota \colon F^R(0) \to \C$. We are now finally ready to state the main result of this section, that rounds off the classification of polynomial exponential functors on $\Vfin$. 
\begin{theorem} \label{thm:classification_part_2}
	Let $W \in \obj{\Vfin}$ and let $R \colon W^{\otimes 2} \to W^{\otimes 2}$ be a unitary involutive $R$-matrix with rescaled Thoma parameters $(0,(b_1, \dots, b_m))$. 	Let $F^R \colon \Vfin \to \Vfin$ be the functor defined in (\ref{eqn:def_of_FR}) and let $\tau_{V_1,V_2}$ and $\iota$ be the natural isomorphisms defined above. Then $(F^R,\tau,\iota)$ is a well-defined polynomial exponential functor. The $R$-matrix associated to $F^R$ is $R$. 
\end{theorem}

\begin{proof}
We have seen in Lemma~\ref{lem:tau_surjective} that $F^R(V)$ is finite-dimensional for every $V \in \obj{\Vfin}$. In particular, $F^R \colon \Vfin \to \Vfin$ is a well-defined polynomial functor with homogeneous components $(F^R_n)_{n \in \N_0}$.

It is straightforward to check that $\widehat{\kappa}^n_{V,0}$ coincides with the canonical homomorphism $(W \otimes V)^{\otimes n} \otimes \C \to (W \otimes (V \oplus 0))^{\otimes n}$ and likewise for $\widehat{\kappa}^n_{0,V}$. This implies that $\iota$ makes the diagram in Def.~\ref{def:exp_functor}~c) commute.

It remains to be shown that $\tau_{V_1,V_2}$ satisfies the associativity condition in Def.~\ref{def:exp_functor}~b). For this it suffices to see that 
\begin{equation} \label{eqn:kappa_associative}
	\varphi_{i,j+k}^{1,2\oplus 3} \circ \left(\id{(W \otimes V_1)^{\otimes i} } \otimes \varphi^{2,3}_{j,k}\right) = \varphi_{i+j,k}^{1\oplus 2,3} \circ \left(\varphi_{i,j}^{1,2} \otimes \id{(W \otimes V_3)^{\otimes k} } \right)
\end{equation}
on $((W \otimes V_1)^{\otimes i})^{S_i} \otimes ((W \otimes V_2)^{\otimes j})^{S_j} \otimes ((W \otimes V_3)^{\otimes k})^{S_k}$. The left hand side involves averaging over the coset space $X_{j,k}$ and then over $X_{i,j+k}$. Identify $X_{j,k}$ with the set of representatives $\sigma^{(2)}_s \in S_{j+k} \subset S_{i+j+k}$, where $S_{j,k}$ embeds into $S_{i+j+k}$ as the subgroup permuting the last $j+k$ elements. Likewise, let $\sigma^{(1)}_r \in S_{i+j+k}$ be the representatives for the elements of $X_{i,j+k}$. Let $X_{i,j,k} = S_{i+j+k}/(S_i \times S_j \times S_k)$ and note that 
	\begin{equation} \label{eqn:bijection}
		X_{i,j+k} \times X_{j,k} \to X_{i,j,k} \qquad , \qquad (\sigma^{(1)}_r, \sigma^{(2)}_s) \mapsto \left[\sigma^{(1)}_r \cdot \sigma^{(2)}_s\right]
	\end{equation}
	is a bijection. In particular, $(\sigma_r^{(1)}\,\sigma_s^{(2)})_{r,s}$ is a collection of representatives in $S_{i+j+k}$ for the elements in $X_{i,j,k}$. The right hand side of (\ref{eqn:kappa_associative}) involves averaging first over $X_{i,j}$, then over $X_{i+j,k}$. A similar argument as above yields a bijection $X_{i+j,k} \times X_{i,j} \to X_{i,j,k}$ proving that the product of the corresponding coset representatives yields another set of representatives for the elements in $X_{i,j,k}$. As was already observed above, the maps $\kappa^{p,q}_{U_1,U_2}$ are independent of the choice of representatives. Hence, the two procedures both yield the average over $X_{i,j,k}$ and therefore	 give the same result on $((W \otimes V_1)^{\otimes i})^{S_i} \otimes ((W \otimes V_2)^{\otimes j})^{S_j} \otimes ((W \otimes V_3)^{\otimes k})^{S_k}$.
	
	Let $\tau \in S_2$ be the non-trivial element. The $R$-matrix of $F^R$ is obtained by restricting the action of $S_2$ on $F^R_2(\C^2)$ given by interchanging the two summands to the direct summand $F^R_1(\C) \otimes F^R_1(\C)$, that is embedded into $F^R_2(\C^2)$ via $\varphi_{1,1}$. Observe that $F^R_1(\C) = W \otimes \C \cong W$ by definition and that 
	\(
		\varphi_{1,1} \colon W \otimes W \to (W \otimes \C^2)^{\otimes 2}
	\)
	is given by \[
		\varphi_{1,1}(w_1 \otimes w_2) = \frac{1}{\sqrt{2}} \left( w_1 \otimes e_1 \otimes w_2 \otimes e_2 + \rho_{R \boxtimes T}^{(2)}(\tau)(w_1 \otimes e_1 \otimes  w_2 \otimes e_2) \right)\ ,
	\] 
	where $\{e_1,e_2\} \subset \C^2$ is the standard basis of $\C^2$. The $R$-matrix $T$ interchanges the two factors of $(\C^2)^{\otimes 2}$, i.e.\ it switches $e_1$ and $e_2$ in the above expression. This implies that the $S_2$-action on $W \otimes \C^2$ that interchanges the summands restricts to  $\rho_R^{(2)}(\tau) = R$ on $W \otimes W$.
\end{proof}

\pagebreak
\begin{lemma} \label{lem:F_R_moneq}
Let $F,G \colon \Vfin \to \Vfin$ be polynomial exponential functors. Let $R,S$ be the $R$-matrices associated to $F$ and $G$ respectively. Then 
\begin{align*}
	F \moneq G \quad & \Leftrightarrow \quad R \moneq S\ .
\end{align*}
In particular, $F \moneq F^R$.
\end{lemma}

\begin{proof}
We will first show $F \moneq F^R$ and deduce the first statement from this. There is a natural isomorphism $F_1(\C) \otimes V \cong F_1(V)$ (cf.\ \cite[App.~A, (5.1)]{book:Macdonald}) and throughout this proof we will identify the two spaces. Let $\ell_n(V) \colon F_n(V) \to F^R_n(V)$ be the isomorphism in \eqref{eqn:lin-iso}. Denote by $\Delta^{V}_{k} \colon V \to V^{\oplus k}$ the isometric linear embedding with 
\[
	\Delta^{V}_{k}(x) = \frac{1}{\sqrt{k}}(x,\dots, x)\ .
\]
Let $q_n^V \colon F_n(V) \to (F_1(V)^{\otimes n})^{S_n} \cong F^R_n(V)$ be the self-adjoint projection onto $L_{F_n}(V, \dots, V)^{S_n}$ followed by the identification of $L_{F_n}(V, \dots, V)^{S_n}$ with $(F_1(V)^{\otimes n})^{S_n}$ as  in Lemma~\ref{lem:linearization}. Then we have $\ell_n(V) = q_n^V \circ F_n(\Delta_n^V)$. Consider the isomorphisms 
\begin{align*}
	\left(\tau_{V_1,V_2}^n\right)^{-1} &\colon \bigoplus_{i+j = n} F_i(V_1) \otimes F_j(V_2) \to F_n(V_1 \oplus V_2) \ , \\
	\left(\tau_{V_1,V_2}^{R,n}\right)^{-1} &\colon \bigoplus_{i+j = n} F^R_i(V_1) \otimes F^R_j(V_2) \to F^R_n(V_1 \oplus V_2) \ .
\end{align*}
For fixed $i,j \in \N$ with $i+j = n$ they restrict to isometric embeddings 
\begin{align*}
	\kappa_{i,j} &\colon F_i(V_1) \otimes F_j(V_2) \to F_n(V_1 \oplus V_2) \ ,\\
	\kappa^R_{i,j} &\colon F^R_i(V_1) \otimes F^R_j(V_2) \to F^R_n(V_1 \oplus V_2)\ .	
\end{align*}
To show $F \moneq F^R$ it suffices to prove that the following diagram commutes:
\[
\begin{tikzcd}[column sep=1.3cm, row sep=1cm]
	F_i(V_1) \otimes F_j(V_2) \ar[r, "\kappa_{i,j}"] \ar[d,"\ell_i(V_1) \otimes \ell_j(V_2)" left] & F_n(V_1 \oplus V_2) \ar[d,"\ell_n(V_1 \oplus V_2)"] \\
	F^R_i(V_1) \otimes F^R_j(V_2) \ar[r, "\kappa_{i,j}^{R}" below] & F^R_n(V_1 \oplus V_2)
\end{tikzcd}	
\]
Let $\iota_{i,j}^{\oplus} \colon V_1^{\oplus i} \oplus V_2^{\oplus j} \to (V_1 \oplus V_2)^{\oplus n}$ be the isometric embedding that adds zeroes in the last $(n-i)$, respectively first $(n-j)$, components. Let $N_{i,j} =  \binom{n}{i}$ and let $\sigma_k$ for $k \in \{1,\dots, N_{i,j}\}$ be coset representatives for the elements of $S_n/(S_i \times S_j)$. Analogous to the definition of $\varphi_{i,j}$ in \eqref{eqn:phi_ij} let $\psi_{i,j} \colon V_1^{\oplus i} \oplus V_2^{\oplus j} \to (V_1 \oplus V_2)^{\oplus n}$ be the linear map given by
\[
	\psi_{i,j}(x,y) = \sum_{k=1}^{N_{i,j}} \sigma_k \cdot \iota_{i,j}^{\oplus}(c_1\,x,c_2\,y) 
\]
for $x \in V_1^{\oplus i}$, $y \in V_2^{\oplus j}$, $c_i>0$ and where $S_n$ acts on $(V_1 \oplus V_2)^{\oplus n}$ by permuting the summands. Note that the constants $c_1,c_2$ can be chosen such that $\psi_{i,j} \circ (\Delta_i^{V_1} \oplus \Delta_j^{V_2}) = \Delta^{V_1 \oplus V_2}_n$. Let 
\[
	\kappa_{V_1^{\oplus i}, V_2^{\oplus j}} \colon F_i(V_1^{\oplus i}) \otimes F_j(V_2^{\oplus j}) \to F_n(V_1^{\oplus i} \oplus V_2^{\oplus j})
\]
be defined analogous to $\kappa_{i,j}$, but with $V_1^{\oplus i}$ in place of $V_1$ and $V_2^{\oplus j}$ instead of $V_2$. Let $\widehat{\kappa}_{i,j} = F_n(\psi_{i,j}) \circ \kappa_{V_1^{\oplus i}, V_2^{\oplus j}}$ and consider the diagram
\[
	\begin{tikzcd}
		F_i(V_1) \otimes F_j(V_2) \ar[r,"\kappa_{i,j}" above] \ar[d,"F_i(\Delta_i^{V_1}) \otimes F_j(\Delta_j^{V_2})" left] & F_n(V_1 \oplus V_2) \ar[d,"F_n(\Delta^{V_1 \oplus V_2}_n)"] \\
		F_i(V_1^{\oplus i}) \otimes F_j(V_2^{\oplus j}) \ar[d,"q_i^{V_1} \otimes q_j^{V_2}" left] \ar[r,"\widehat{\kappa}_{i,j}"] & F_n((V_1 \oplus V_2)^{\oplus n}) \ar[d,"q_n^{V_1 \oplus V_2}"] \\
		F^R_i(V_1) \otimes F^R_j(V_2) \ar[r,"\kappa^R_{i,j}" below] & F^R_n(V_1 \oplus V_2)
	\end{tikzcd}
\]
The naturality of $\kappa_{i,j}$ and our choice for $c_1,c_2$ implies that the upper square commutes. Restricted to  $(F_1(V_1 \oplus V_2)^{\otimes n})^{S_i \times S_j}$ the map $F_n(\sum_{k} \sigma_k)$ corresponds to a multiple of the orthogonal projection to $(F_1(V_1 \oplus V_2)^{\otimes n})^{S_n}$, while $F_n(\iota^{\oplus}_{i,j}(c_1 \cdot, c_2 \cdot))$ restricts to a scalar multiple of $\iota_{i,j}$ on $F_i^R(V_1) \otimes F_j^R(V_2)$. Hence, the lower square commutes up to a scalar. But since the outer square is a diagram of isometric embeddings, this scalar has to be $1$. 

Thus, to prove the first statement it suffices to see that $F^R \moneq F^S$ implies $R \moneq S$ and vice versa. As was shown in the proof of Lemma~\ref{lem:F_R_equivalences} the unitary isomorphism $v_n(
\C^{\oplus n}) \colon F^R_n(\C^{\oplus n}) \to F^S_n(\C^{\oplus n})$ restricts to a unitary intertwiner $v^{R,S}_n \colon W_R^{\otimes n} \to W_S^{\otimes n}$ between the representations $\rho_R^{(n)}$ and $\rho_S^{(n)}$, where $W_R = F^R_1(\C)$ and $W_S = F^S_1(\C)$. Let $\kappa^R_{i,j}, \kappa^S_{i,j}$ be the isometric embeddings for $R$ and $S$, respectively. Let $c_{i,j} \colon W_R^{\otimes i} \otimes W_R^{\otimes j} \to W_R^{\otimes n}$ be the canonical isomorphism. The associativity condition (Def.~\ref{def:exp_functor} b)) implies that the following diagram commutes:
\[
	\begin{tikzcd}
		F^R_i(\C^{\oplus i}) \otimes F^R_j(\C^{\oplus j}) \ar[r,"\kappa_{i,j}^R"] & F^R_n(\C^{\oplus n}) \\
		W_R^{\otimes i} \otimes W_R^{\otimes j} \ar[u] \ar[r,"c_{i,j}" below] & W_R^{\otimes n} \ar[u]
	\end{tikzcd}
\]
where the unlabelled vertical arrows identify $W_R^{\otimes k}$ with $L_{F^R_k}(\C,\dots, \C)$ and embed it into $F_k^R(\C^{\oplus k})$. Together with the naturality of $v$ this implies that $v_n^{R,S} = v_i^{R,S} \otimes v_j^{R,S}$ and thus $v_n^{R,S} = v_1^{R,S} \otimes \dots \otimes v_1^{R,S}$. Hence, $R \moneq S$.

In the other direction let $v_1 \colon W_R \to W_S$ be a unitary isomorphism  witnessing the equivalence $R \moneq S$. It induces a unitary isomorphism
\[
	\hat{v}_i(V) \colon (W_R \otimes V)^{\otimes i} \to (W_S \otimes V)^{\otimes i}
\]
by letting $v_1^{\otimes i}$ act on the tensor factor $W_R^{\otimes i}$. This makes the following diagram commute:
\[
	\begin{tikzcd}
		(W_R \otimes V_1)^{\otimes i} \otimes (W_R \otimes V_2)^{\otimes j} \ar[r,"\varphi^R_{i,j}"] \ar[d,"\hat{v}_i(V_1) \otimes \hat{v}_j(V_2)" left] & (W_R \otimes (V_1 \oplus V_2))^{\otimes n} \ar[d,"\hat{v}_1(V_1 \oplus V_2)"] \\
		(W_S \otimes V_1)^{\otimes i} \otimes (W_S \otimes V_2)^{\otimes j} \ar[r,"\varphi^S_{i,j}" below] & (W_S \otimes (V_1 \oplus V_2))^{\otimes n}
	\end{tikzcd}
\]
Here we use the fact that $\hat{v}_n(V_1 \oplus V_2)$ restricts to $\hat{v}_i(V_1) \otimes \hat{v}_j(V_2)$ on the direct summand $(W_R \otimes V_1)^{\otimes i} \otimes (W_R \otimes V_2)^{\otimes j}$. Taking the sum over the homogeneous components we see that the restriction $v_n(V)$ of $\hat{v}_n(V)$ to the $S_n$-invariant subspace $((W_R \otimes V)^{\otimes n})^{S_n} = F^R(V)$ induces a monoidal equivalence between $F^R$ and $F^S$. This proves the statement.  
\end{proof}

\section{Twists via Exponential Functors} 
In this section we will study a natural map $\tau_F \colon BBU_{\oplus} \to B\BUloc{d}$ induced by an exponential functor $F \colon \Viso \to \Viso$ with $d = \dim(F(\C))$. Combined with the equivalence $SU \simeq BBU_{\oplus}$  obtained from Bott periodicity this yields the twists 
\(
	\tau^n_F \colon SU(n) \to B\BUloc{d}
\)
alluded to in the introduction. 

We will then study how they behave in rationalised $K$-theory $KU_{\Q}$. Since the Chern character identifies $KU_{\Q}$ with ordinary cohomology in even degrees, the twists correspond to odd degree cohomology classes. To compute these for $\tau^n_F$ we will employ maps $S \C P^{n-1} \to SU(n)$ detecting the generators of $H^*(SU(n),\Z)$. 

\subsection{The twist induced by an exponential functor} \label{sec:twists_via_mon_cats}
Let $\Viso$ be the isomorphism subcategory of $\Vfin$ and note that $\Viso$ is essentially small. A skeleton is constructed as follows: Let $\BUpl$ be the subcategory of $\Viso$ with objects $\N_0$ (with $k \in \N_0$ corresponding to $\C^k$) and
\[
	\hom(m,n) = \begin{cases}
		\emptyset & \text{if } m \neq n \\
		U(n) & \text{otherwise}
	\end{cases}\ .
\]
The geometric realisation of the nerve of this category is clearly given by 
\[
	\lvert \BUpl \rvert = \coprod_{n \in \N_0} BU(n)\ .
\]
Moreover, $\BUpl$ has a monoidal structure given by the addition of natural numbers on objects and the block sum on morphisms. This reflects the monoid structure on isomorphism classes of vector bundles given by direct sum. On the geometric realisation it corresponds to maps 
\[
	BU(n) \times BU(m) \to BU(n + m)
\] 
induced by the group homomorphism $U(n) \times U(m) \to U(n+m)$ given by the block sum. By \cite[Cor.~11.7]{book:MayIteratedLoopSpaces} the space $\lvert \BUpl \rvert$ is a topological monoid and its group completion is 
\[
\Omega B\lvert \BUpl \rvert \simeq BU_{\oplus} \times \Z \ ,
\]
where $BU_{\oplus}$ is the classifying space of the colimit $U$ of the unitary groups $U(n)$ with respect to the inclusions $U(n) \to U(n+1)$. Note that the classifying space $BU$ has two natural $H$-space structures corresponding to the direct sum and the tensor product of vector bundles, respectively. We will decorate $BU$ accordingly, if we want to stress the $H$-space structure we have in mind.  

Given an exponential functor $F \colon \Viso \to \Viso$ let $W = F(\C)$ and let $\BUte{W}$ be the category with objects $\N_0$ and  
\[
	\hom(m,n) = \begin{cases}
		\emptyset & \text{if } m \neq n \\
		U(W^{\otimes n}) & \text{otherwise}
	\end{cases}\ ,
\]
where $U(W^{\otimes n})$ is the group of unitary automorphisms on the inner product space $W^{\otimes n}$. This is a monoidal category. The tensor product is given by the addition of natural numbers on the objects and by the tensor product of unitaries
\[
\begin{tikzcd}
	U(W^{\otimes m}) \times U(W^{\otimes n}) \ar[r,"\otimes"] & U(W^{\otimes (m+n)}) \quad , \quad (u_1,u_2) \mapsto u_1 \otimes u_2
\end{tikzcd}
\]
on the morphisms. Note that we understand 
\[
	W^{\otimes n} = (\dots (W \otimes W) \otimes W) \otimes \dots )\ ,
\] 
i.e.\ we fix a choice of brackets for the tensor powers of $W$. In particular, the associators of $\BUte{W}$ are identities. The maps $W^{\otimes m} \otimes W^{\otimes n} \to W^{\otimes n} \otimes W^{\otimes m}$, which interchange the tensor factors, turn $\BUte{W}$ into a permutative category via conjugation. 
The geometric realisation of this category is 
\[
	\lvert \BUte{W} \rvert = \coprod_{n \in \N_0} BU(W^{\otimes n})\ .
\]
Let $d = \dim(W)$ and denote by $1_d \in U(W)$ the identity matrix. The tensor product described above induces an embedding of $U(W^{\otimes n})$ into $U(W^{\otimes (n+1)})$ defined by  
\[
	\iota_n \colon U(W^{\otimes n}) \to U(W^{\otimes (n+1)}) \quad , \quad u \mapsto u \otimes 1_d\ ,
\]
which is a group homomorphism. 

Let $\Zloc{d} \subset \Q$ be the subring of the rationals given by all fractions $\tfrac{p}{q}$ with $q$ an arbitrary power of $d$. By \cite[Thm.~2.2]{paper:Sullivan} there is a topological space $\BUloc{d}$ and a continuous map $\ell \colon BU \to \BUloc{d}$, such that the pair is a localisation of the simply connected space $BU$ in the sense of \cite[Def.~2.1]{paper:Sullivan} with respect to the set of prime factors of $d$, i.e.\ with $\Z_{\ell} = \Zloc{d}$. For further details we refer the reader to \cite[Chap.~3]{book:Adams}. The following is a well-known result from homotopy theory. Since we could not find a proof in the literature, we give one here.

\begin{lemma} \label{lem:htpy_type_BUF}
	The group completion of $\lvert \BUte{W} \rvert$ satisfies
	\[
		\Omega B \lvert \BUte{W} \rvert \simeq \BUloc{d} \times \Z \ .
	\]
\end{lemma}

\begin{proof}
The homomorphism $\iota_n$ induces $B\iota_n \colon BU(W^{\otimes n}) \to BU(W^{\otimes (n+1)})$. The collection of all $B\iota_n$ gives a continuous map $B\iota_{\infty} \colon \lvert\BUte{W}\rvert \to \lvert\BUte{W}\rvert$.
Let $\lvert \BUte{W}\rvert_{\infty}$ be the mapping telescope of the sequence
\[
	\begin{tikzcd}
		\lvert \BUte{W}\rvert \ar[r,"B\iota_{\infty}"] & \lvert\BUte{W}\rvert \ar[r,"B\iota_{\infty}"] & \lvert\BUte{W}\rvert \ar[r,"B\iota_{\infty}"] & \dots
	\end{tikzcd}
\] 
Given a finite dimensional inner product space $V$ denote by $U_{\infty}(V)$ the colimit over $m$ of $U(\bigoplus_{k=1}^m V) \to U(\bigoplus_{k=1}^{m+1} V)$ defined by $w \mapsto w \oplus \id{V}$. The embedding into the upper left hand corner induces a continuous map $BU(V) \to BU_{\infty}(V)$. The following diagram commutes:
\[
	\begin{tikzcd}[column sep=1.9cm]
		BU(W^{\otimes n}) \ar[r,"B(u\, \mapsto u \otimes 1_d)" below] \ar[d] & BU(W^{\otimes (n+1)}) \ar[d]\\
		BU_{\infty}(W^{\otimes n}) \ar[r,"B(u\, \mapsto u \otimes 1_d)" below] & BU_{\infty}(W^{\otimes (n+1)}) 
	\end{tikzcd}
\]
The map $BU(W^{\otimes n}) \to BU(W^{\otimes (n+1)})$ induced by $u \mapsto u\otimes 1_d$ classifies the tensor product of the universal bundle with the trivial vector bundle with fibre $W$, which has rank $d$. It induces multiplication by $d$ on all homotopy groups $\pi_k$ with $k > 0$. The analogous statement is true for the bottom line of the diagram. Thus, the vertical maps induce a weak equivalence of the respective colimits. Each space $BU_{\infty}(W^{\otimes n})$ is homotopy equivalent to $BU$. By \cite[Prop.~6.61]{book:MayPonto} the colimit over the maps in the last row is equivalent to the localisation $\BUloc{d}$. This gives the homotopy type of the path components of $\lvert \BUte{W}\rvert$. Furthermore $\pi_0(\Omega B\lvert \BUte{W}\rvert) \cong \Z$. As a consequence we obtain
\[
	\lvert \BUte{W}\rvert_{\infty} \simeq \BUloc{d} \times \Z\ .
\]
Note in particular that each path component is simply connected. Since $\BUte{W}$ is permutative, the topological monoid $\lvert \BUte{W}\rvert$ is homotopy commutative. The statement now follows from \cite[Cor.~1.2]{paper:Randal-Williams} and the observation that  $\lvert \BUte{W}\rvert_{\infty}^+ \simeq \lvert \BUte{W}\rvert_{\infty}$.
\end{proof}

\begin{remark} \label{rem:Z-component}
	The path components of $\lvert \BUte{W} \rvert$ are labelled by natural numbers, where $k \in \N$ corresponds to the space $BU(W^{\otimes k})$ classifying vector bundles of rank $d^k$. Thus, we should think of the path component of $\Omega B \lvert \BUte{W} \rvert$ labelled by $k \in \Z$ as the one classifying ``localised'' virtual vector bundles of virtual dimension $d^k$ and hence of the factor $\Z$ as sitting inside $\Zloc{d}$ via the embedding $k \mapsto d^k$.   
\end{remark}

The exponential functor $F$ induces homomorphisms $ U(n) \to U(F(\C^n))$. Let $\tau_{m,n} = \tau_{\C^m, \C^n}$ be the natural transformation that is part of the structural data of $F$. By splitting off tensor factors via $\tau_{1,k}$ we obtain a unitary isomorphism 
\[
	\tau_n \colon F(\C^n) \to W^{\otimes n}\ .
\]
The associativity of $\tau_{\C^m,\C^n}$ implies that the following diagram commutes
\[
	\begin{tikzcd}
		F(\C^m \oplus \C^n) \ar[r,"\tau_{m,n}"] \ar[d,"\tau_{m+n}" left] & F(\C^m) \otimes F(\C^n) \ar[d,"\tau_m \otimes \tau_n"]\\
		W^{\otimes (m + n)} \arrow[r,equal] & W^{\otimes m} \otimes W^{\otimes n}
	\end{tikzcd}
\]
Moreover, we have 
\[
	\tau_{m,n} \circ F(u_1 \oplus u_2) = F(u_1) \otimes F(u_2) \circ \tau_{m,n}
\]
by naturality of $\tau_{m,n}$. Define $\kappa_n \colon U(n) \to U(W^{\otimes n})$ by $u \mapsto \tau_n \circ F(u) \circ \tau_n^{-1}$. For $u_1 \in U(m)$ and $u_2 \in U(n)$ we obtain
\begin{align}
	\kappa_{m+n}(u_1 \oplus u_2) &= (\tau_m \otimes \tau_n) \circ \tau_{m,n} \circ F(u_1 \oplus u_2) \circ \tau_{m,n}^{-1} \circ (\tau_m \otimes \tau_n)^{-1} \\
	&= (\tau_m \circ F(u_1) \circ \tau_m^{-1}) \otimes (\tau_n \circ F(u_2)\circ \tau_n^{-1}) \notag \\
	&= \kappa_m(u_1) \otimes \kappa_n(u_2)\ . \notag
\end{align}
This implies that the collection of all $\kappa_n$ combines to a (strict) monoidal functor $\kappa \colon \BUpl \to \BUte{W}$.
Thus, we obtain a continuous monoid homomorphism $\lvert \kappa \rvert \colon \lvert \BUpl \rvert \to \lvert \BUte{W} \rvert$, which induces a map $B\lvert \kappa \rvert \colon B\lvert \BUpl \rvert \to B\lvert \BUte{W} \rvert$ that restricts to 
\[
	\tau_F \colon B(BU_{\oplus} \times \{0\}) \to B(\BUloc{d} \times \{0\})\ .
\]
The codomain is equivalent to the $1$-connected cover of $BGL_1\!\left(KU\left[\tfrac{1}{d}\right]\right)$ classifying twists of localised complex $K$-theory. 

\begin{remark} \label{rem:ext_power}
In the construction it suffices to have an exponential functor $F \colon \Viso \to \Viso$ on the isomorphism groupoid. An important example of a functor of this type without an extension to $\Vfin$ is the top exterior power functor $\extp^{\rm top}$. The map $BU(n) \to BU(F(\C^n)) \simeq BU(1)$ induced by $\extp^{\rm top}$ agrees with the one induced by the determinant and the corresponding map $\tau_{\det} \colon BBU_{\oplus} \to BBU_{\otimes}$ factors as
\[
	\tau_{\det} \colon BBU_{\oplus} \to BBU(1) \to BBU_{\otimes}\ .
\] 
\end{remark}

\subsection{Detecting the cohomology of $SU(n)$ via $\C P^n$}
In this section we will construct maps $f_n \colon S \C P^{n-1} \to SU(n)$ detecting the generators in cohomology. This will be applied to compute the rational characteristic classes of the twists $\tau_F$ in Sec.~\ref{sec:twist_over_SUn}. Fix $n \in \N$, let $e \in M_n(\C)$ be the rank~$1$-projection onto the first component of $\C^n$ and let $\omega \colon I \to SU(n)$ be given by
\[
 	\omega(t) = (I_n - (1 - \exp(2\pi i t))e) \exp\left(-2\pi i \tfrac{t}{n}\right) \ ,
\] 
where $I_n \in M_n(\C)$ denotes the identity matrix. Now consider the map $\widehat{f}_n \colon I \times SU(n) \to SU(n)$ defined by
\[
	\widehat{f}_n(t,u) = u\,\omega(t)\,u^*
\]
Observe that $\widehat{f}_n(t,u)$ only depends on the class of $u$ in the homogeneous space $SU(n)/S(U(1) \times U(n-1))$, which is homeomorphic to $\C P^{n-1}$ via the map that sends $[v] \in SU(n)/S(U(1) \times U(n-1))$ to $[x_1: \dots:x_n] \in \C P^{n-1}$ with $x = ve_1$. Moreover, $\omega(0) = I_n$ and $\omega(1) = \exp\!\left(2 \pi i \tfrac{1}{n}\right)$. Hence, $\widehat{f}_n$ induces a well-defined map 
\begin{equation} \label{eqn:map_fn}
	f_n \colon S \C P^{n-1} \to SU(n)
\end{equation}
from the (unreduced) suspension of $\C P^{n-1}$ to $SU(n)$. 

The cohomology rings of these spaces are well-known: For $SU(n)$ we obtain an exterior algebra over $\Z$ on generators in odd degrees between $3$ and $2n-1$, i.e.\ 
\[
	H^*(SU(n),\Z) \cong \Lambda_{\Z}[a_3,a_5, \dots, a_{2n-1}]\ .
\] 
Up to a sign the generators $a_{2k+1} \in H^{*}(SU(n),\Z)$ are obtained inductively as follows: Let $p_{SU(n)} \colon SU(n) \to S^{2n-1}$ be defined by $p_{SU(n)}(v) = ve_n$. This is a fibration with fibre $SU(n-1)$ to which the Leray-Hirsch theorem applies. We obtain the isomorphism 
\[
	H^*(SU(n),\Z) \cong H^*(SU(n-1),\Z) \otimes H^*(S^{2n-1},\Z)
\]
For $n = 2$ the map $p_{SU(2)}$ is a homeomorphism and choosing a generator of $H^3(S^3,\Z)$ fixes $a_3 \in H^3(SU(2),\Z)$ as well. For $n > 2$, the generators $a_3, \dots, a_{2n-3} \in H^*(SU(n),\Z)$ are all fixed by the fact that the restriction to $SU(n-1)$ maps them to the corresponding generators there. The remaining element $a_{2n-1} \in H^{2n-1}(SU(n),\Z)$ is the image of a generator of $H^{2n-1}(S^{2n-1},\Z)$ under $p_{SU(n)}^* \colon H^{2n-1}(S^{2n-1},\Z) \to H^{2n-1}(SU(n),\Z)$.

We will also need the cohomology ring of $U(n)$, which is a similar exterior algebra, but with one more generator in degree~$1$:
\[
	H^*(U(n),\Z) \cong \Lambda_{\Z}[c_1,c_3, \dots, c_{2n-1}]\ .
\] 
The cohomology ring of $\C P^{n-1}$ is a truncated polynomial ring on one generator $t$ in degree $2$, i.e.\ 
\[
	H^*(\C P^{n-1}, \Z) \cong \Z[t]/(t^n)\ ,
\]
where the element $t \in H^2(\C P^{n-1},\Z)$ is the Chern class of the canonical line bundle over $\C P^{n-1}$. On cohomology the inclusion map $\iota \colon SU(n) \to U(n)$ sends $c_1$ to $0$ and identifies $c_{2k+1}$ with $a_{2k+1}$ for $k \in \{1, \dots, n-1\}$. Thus, to understand $f_n^* \colon H^*(SU(n),\Z) \to H^*(S\C P^{n-1},\Z)$ it suffices to consider
\[
	\iota \circ f_n \colon S\C P^{n-1} \to U(n)\ .
\]
However, this map is homotopic to 
\begin{equation} \label{eqn:def_of_hn}
	h_n \colon S \C P^{n-1} \to U(n) \quad , \quad h_n([t,u]) = I_n - (1 - \exp(2\pi i t))u e u^*
\end{equation}
The class $c_{2n-1} \in H^{2n-1}(U(n),\Z)$ is the image of a generator under the homomorphism induced by $p = p_{U(n)} \colon U(n) \to S^{2n-1}$ with $p(v) = ve_n$. We define $g_n = p \circ h_n \colon S \C P^{n-1} \to S^{2n-1}$. The inclusion $SU(n-1) \to SU(n)$ into the upper left hand corner induces an embedding $S\C P^{n-2} \to S\C P^{n-1}$. With respect to the homeomorphism $SU(n)/S(U(1) \times U(n-1)) \cong \C P^{n-1}$ described above this map corresponds to the suspension of $\C P^{n-2} \subset \C P^{n-1}$ which identifies $[x_1 : \dots : x_{n-1}]$ with $[x_1 : \dots : x_{n-1} : 0]$. The effect of $g_n$ on the top-dimensional cell of $S\C P^{n-1}$ was described in \cite[Lem.~6.1]{paper:Yokota} with the following result: 

\begin{lemma} \label{lem:gn_homeomorphism}
	The map $g_n$ induces a homeomorphism 
	\[
		\overline{g}_n \colon S \C P^{n-1}/S \C P^{n-2} \to S^{2n-1}\ .
	\]
\end{lemma}

Let $\sigma \colon H^k(SX,\Z) \to H^{k-1}(X,\Z)$ for $k > 1$ be the suspension isomorphism. We are now ready to prove that the homomorphisms $f_n^*$ together with the suspension isomorphisms detect the generators of $H^*(SU(n),\Z)$. 
\begin{lemma} \label{lem:map_detects_coh_classes}
	The map $f_n$ defined in \eqref{eqn:map_fn} induces the following homomorphism on cohomology  
	\[
		\sigma \circ f^*_n \colon H^{2k+1}(SU(n),\Z) \to H^{2k}(\C P^{n-1},\Z)\ .
	\] 
	The generators $a_{2k+1} \in H^{*}(SU(n),\Z)$ can be chosen in such a way that the above homomorphism satisfies 
	\[
		(\sigma \circ f_n^*)(a_{2k+1}) = t^{k}
	\]
	for $k \in \{1, \dots, n-1\}$. 
\end{lemma}

\begin{proof}
We will prove the statement by induction over $n$. For $n = 2$ the map $p_{SU(2)}$ is a homeomorphism and $p_{SU(2)} \circ f_2$ is homotopic to $g_2 \colon S\C P^{1} \to S^3$, which in turn is a homotopy equivalence by Lemma~\ref{lem:gn_homeomorphism} and the contractibility of $S\C P^0 \cong I$. Thus, if $x_{3} \in H^3(S^3,\Z)$ is a generator, then $\sigma \circ f_2^* \circ p_{SU(2)}^*(x_3)$ has to agree with $\pm t \in H^2(\C P^1,\Z) \cong \Z$, but $p_{SU(2)}^*(x_3) = a_3$.

Let $\iota \colon SU(n) \to U(n)$ be the inclusion map. As was already pointed out above we have $\iota^*(c_{2k+1}) = a_{2k+1}$ for $k \in \{1, \dots, n-1\}$ and 
	\(
		f_n^* \circ \iota^* = h_n^*
	\).
	Hence, it suffices to show that $h_n^*(c_{2k+1}) = t^k$. Assume that we have proven this statement in all dimensions $m \in \{2, \dots, n-1\}$. To prove that it follows for $m = n$ as well note that the following diagram commutes
	\[
		\begin{tikzcd}
			S\C P^{n-2} \ar[d,"h_{n-1}" left] \ar[r,"j_{n-1}"]  & S \C P^{n-1} \ar[d,"h_n" right] \\
			U(n-1)  \ar[r, "\iota_n" below] & U(n) 
		\end{tikzcd}
	\]
	where $j_{n-1}$ and $\iota_n$ denote the inclusions described above. Observe that $\iota_n \colon U(n-1) \to U(n)$ induces the homomorphism 
	\begin{gather*}
		\iota_n^* \colon \Lambda_{\Z}[c_1, \dots, c_{2n-3}, c_{2n-1}] \to \Lambda_{\Z}[c_1', \dots, c_{2n-3}'] \ ,\\
		\iota_n^*(c_{2k-1}) = 
		\begin{cases}
			c_{2k-1}' & \text{if } 1 \leq k \leq n-1 \\
			0 & \text{if } k = n
		\end{cases}\ .
	\end{gather*}
	whereas $j_{n-1}^*$ can be identified with the quotient homomorphism 
	\[
		j_{n-1}^* \colon \Z[t]/(t^{n}) \to \Z[t']/((t')^{n-1}) \ .
	\] 
	For $k \in \{1, \dots, n-2\}$ our induction hypothesis implies 
	\[
		(t')^k = h^*_{n-1}(c_{2k+1}') = h^*_{n-1}(\iota_n^*(c_{2k+1})) = j^*_{n-1}(h_n^*(c_{2k+1}))
	\]
	and therefore $h_n^*(c_{2k+1}) = t^k$. Hence, it only remains to be seen that we can choose the generators $c_{2n-1}$ such that $h_n^*(c_{2n-1}) = t^{n-1}$. It suffices to prove that $h_n^* \circ p^* = g_n^*$ is an isomorphism. The long exact sequence of the pair $(S\C P^{n-1},S\C P^{n-2})$ shows that $H^{2n-1}(S \C P^{n-1}, S \C P^{n-2},\Z) \to H^{2n-1}(S \C P^{n-1},\Z)$ is an isomorphism. The domain is isomorphic to the cohomology of the quotient and the commutative diagram
	\[
		\begin{tikzcd}
			H^{2n-1}(S^{2n-1},\Z) \ar[d,"\overline{g}_n^*" left] \ar[equal]{r} & H^{2n-1}(S^{2n-1},\Z) \ar[d,"g_n"] \\
			H^{2n-1}(S \C P^{n-1}/S \C P^{n-2},\Z) \ar[r,"\cong" below] &H^{2n-1}(S \C P^{n-1},\Z)
		\end{tikzcd}
	\]
	together with Lemma~\ref{lem:gn_homeomorphism} proves that $g_n$ is indeed an isomorphism.
\end{proof}

\subsection{Twists over $SU(n)$ from Bott periodicity} \label{sec:twist_over_SUn}
Bott periodicity yields a homotopy equivalence \[ B(BU_{\oplus} \times \Z) \simeq U\ .\] The map $B(BU_{\oplus} \times \Z) \to B\Z$ induced by the projection corresponds to the determinant map on the right hand side. Therefore, Bott periodicity restricts to an equivalence $BBU_{\oplus} = B(BU_{\oplus} \times \{0\}) \simeq SU$ of the corresponding homotopy fibres. The group $SU$ is the colimit over all $SU(n)$ over the embeddings $SU(n) \to SU(n+1)$. In particular, we obtain twists of the form
\begin{equation} \label{eqn:twists_SU(n)}
\begin{tikzcd}
	\tau_F^n \colon SU(n) \ar[r] & SU \simeq BBU_{\oplus} \ar[r,"\tau_F"] & B\BUloc{d} 
\end{tikzcd}
\end{equation}
These twists over $SU(n)$ provide a (non-equivariant) generalisation of the one described by the basic gerbe \cite{paper:Meinrenken, paper:FreedHopkinsTelemanI, paper:MurrayStevenson} in the following sense: Let $\tau_{\det}$ be the map associated to the exponential functor $\extp^{\rm top} \colon \Viso \to \Viso$ as in Rem.~\ref{rem:ext_power}. As we will see below, the associated twist $\tau^n_{\det}$ corresponds to a generator $a_3 \in H^3(SU(n),\Q) \cong H^3_{SU(n)}(SU(n),\Q) \cong \Q$ and therefore agrees with the one given by the basic gerbe. 

The rationalisation of the generalised cohomology theory associated to the infinite loop space $BBU_{\otimes}$ is well-understood \cite[Prop.~3.5]{paper:Teleman}. Up to equivalence it agrees with rational ordinary cohomology in odd degrees starting with $3$, i.e.
\begin{align} \label{eqn:rational_equiv}
	\left(BBU_{\otimes}\right)_{\Q} \simeq \prod_{n \in \N} K(2n+1,\Q)\ .
\end{align}
To obtain a map inducing the above equivalence note that the Chern character provides a $E_{\infty}$-ring map that identifies rational topological $K$-theory $KU_{\Q}$ with even rational cohomology $H\Q$. It restricts to an equivalence $BGL_1(KU_{\Q}) \to BGL_1(H{\Q})$. Restricting to the $1$-connected cover and composing this transformation with the natural logarithm on rational cohomology \cite[Sec.~2.5, Prop.~2.6]{paper:Rezk}, we obtain the equivalence \eqref{eqn:rational_equiv}.

Composing \eqref{eqn:twists_SU(n)} with the natural map $B\BUloc{d} \to (BBU_{\otimes})_{\Q}$ and then applying \eqref{eqn:rational_equiv} we obtain a class
\begin{align} \label{eqn:delta}
	\delta^{F,n} = \delta^{F,n}_3 + \delta^{F,n}_5 + \delta^{F,n}_7 + \dots  \quad \in H^{\text{odd}}(SU(n),\Q)\ ,
\end{align}
which can be expressed in terms of the generators $a_{2k+1} \in H^{2k+1}(SU(n),\Q)$. We will employ Lemma~\ref{lem:map_detects_coh_classes} to make this explicit. Hence, we need to understand $f_n^*(\delta^{F,n}) \in H^{\text{odd}}(S \C P^{n-1}, \Q) \cong H^{\text{even}}(\C P^{n-1}, \Q)$. A first step in this direction is provided by the next lemma.

\begin{lemma} \label{lem:classifying_map_plus}
Let $f_n \colon S \C P^{n-1} \to SU(n)$ be the map defined in \eqref{eqn:map_fn}. Let $\alpha_n \colon \C P^{n-1} \to BU_{\oplus}$ be the continuous map that is adjoint to 
\[
	\begin{tikzcd}
		\Sigma \C P^{n-1} \simeq S\C P^{n-1} \ar[r,"f_n"] & SU(n) \ar[r] & SU \simeq BBU_{\oplus}
	\end{tikzcd}
\] 
Then $\alpha_n$ classifies the virtual vector bundle $[H] - [\C]$, where $H$ is the tautological line bundle over $\C P^{n-1}$. 
\end{lemma}

\begin{proof}
	Let $\iota \colon SU(n) \to U(n)$ and $\iota_{\infty} \colon SU \to U$ be the inclusion maps. Let $h_n \colon S\C P^{n-1} \to U(n)$ be the map defined in \eqref{eqn:def_of_hn}. In the following diagram the triangle commutes up to homotopy, while the square commutes.
	\[
		\begin{tikzcd}
			\Sigma \C P^{n-1} \simeq S\C P^{n-1} \ar[r,"f_n"] \ar[dr,"h_n" below] & SU(n) \ar[r,"\alpha"] \ar[d,"\iota" left]  & SU \ar[d,"\iota_{\infty}"] \\
			 & U(n) \ar[r, "\alpha" below]  & U 
		\end{tikzcd}
	\]
	This implies that the two maps $\C P^{n-1} \to \Omega U$ adjoint to $\iota_{\infty} \circ \alpha \circ f_n$ and $\alpha \circ h_n$, respectively, agree up to homotopy. Moreover, the homotopy equivalences $\Omega SU \simeq BU_{\oplus}$ and $\Omega U \simeq BU_{\oplus} \times \Z$ identify $\Omega \iota_{\infty} \colon \Omega SU \to \Omega U$ with the inclusion map $BU_{\oplus} \to BU_{\oplus} \times \Z$ onto the $0$-component. Hence it remains to be proven that the map adjoint to $\alpha \circ h_n$ classifies the stable isomorphism class of $H$ in reduced $K$-theory, which corresponds to $[H] - [\C]$, when $\widetilde{K}^0(\C P^{n-1})$ is viewed as formal differences of unreduced classes.
	
	To see why this is true, we have to take a look at the equivalence $U \simeq B(\coprod_{n \in \N_0} BU(n))$. We follow the proof of Bott periodicity given in \cite{paper:Harris}. Let $\K$ be the compact operators on a separable infinite dimensional Hilbert space~$H$ and denote by $\mathcal{P}(\K)$ the space of all projections in $\K$ with the norm topology. Let $\lvert X_{\bullet}\rvert$ be the geometric realisation of the simplicial space $(X_n)_{n \in \N_0}$ with $X_0 = \{\text{pt}\}$ and
	\[
		X_n = \left\{ (p_1, \dots, p_n) \in \mathcal{P}(\K)^n\ |\ p_ip_j = 0 \text{ for all } i \neq j \right\}
	\]
	for $n > 0$, where the face maps take sums of projections and the degeneracy maps add a  zero projection to the tuple. The spectral decomposition of elements in $U$ yields a homeomorphism \cite[Sec.~2, p.~451]{paper:Harris}
	\[
		s \colon \lvert X_{\bullet} \rvert \to U \quad, \quad [(t_1, \dots, t_n), (p_1, \dots, p_n)] \mapsto \exp\left( 2 \pi i \sum_{i=1}^n t_ip_i \right)\ .
	\]
	Let $\Proj{e}{\K} \subset \mathcal{P}(\K)$ be the subspace of rank $1$ projections. Note that this has the homotopy type of the classifying space $BU(1)$. Let $j \colon S\Proj{e}{\K} \to \lvert X_{\bullet} \rvert$ be the canonical map that identifies $[t,p] \in S\Proj{e}{\K}$ with the corresponding point in $\lvert X_{\bullet}\rvert$. 	The map $h_n$ can be rewritten as follows
	\[
		h_n([t,u]) = \exp(2\pi i t\,ueu^*)\ .
	\]
	Thus, if we define $k_n \colon \C P^{n-1} \to \Proj{e}{\K}$ by $k_n([u]) = ueu^* \in \Proj{e}{\K} \subset X_1$, then the following diagram commutes
	\[
		\begin{tikzcd}
			S \C P^{n-1} \ar[d,"Sk_n" left] \ar[r,"\alpha \circ h_n"] & U \\
			S \Proj{e}{\K} \ar[r,"j" below] & \lvert X_{\bullet} \rvert \ar[u,"s" right, "\cong" left]
		\end{tikzcd}
	\]
	The proof of Bott periodicity in \cite{paper:Harris} proceeds by constructing two simplicial spaces $(Z_n)_{n \in \N_0}$ and $(Y_n)_{n \in \N_0}$ and equivalences 
	\[
		\begin{tikzcd}
			U \cong \lvert X_{\bullet}\rvert & \ar[l,"\simeq" above] \lvert Z_{\bullet}\rvert \ar[r,"\simeq"] & \lvert Y_{\bullet}\rvert \simeq B\left(\coprod_{n \in \N_0} BU(n)\right)\ .
		\end{tikzcd}
	\]
	Tracing the map $j$ through these equivalences we see that it agrees up to homotopy with
	\[
		\begin{tikzcd}
			\bar{j} \colon SBU(1) \ar[r,"\circled{1}"] & \Sigma BU(1)_+ \ar[r,"\circled{2}"] &  B \left(\coprod_{n \in \N_0} BU(n) \right) 
		\end{tikzcd}
	\] 
	where $BU(1)_+$ denotes the space $BU(1)$ with a disjoint basepoint, $\circled{\footnotesize{1}}$ is the quotient map and $\circled{\footnotesize{2}}$ is induced by the inclusion of the $1$-skeleton. Thus, the bundle classified by the map adjoint to $\Sigma \C P^{n-1} \simeq S\C P^{n-1} \to B(\coprod_{n \in \N_0} BU(n))$ agrees with the one classified by $k_n \colon \C P^{n-1} \to \Proj{e}{\K}$. Up to the homotopy equivalences described above this agrees with the quotient map
	\[
		U(n)/U(1) \times U(n-1) \to U(H)/U(1) \times U(e_1^{\perp})
	\]
	classifying the tautological bundle $H \cong U(n)/U(n-1) \to \C P^{n-1}$.
\end{proof}

\begin{corollary} \label{cor:classifying_map_tensor}
	Let $f_n \colon S \C P^{n-1} \to SU(n)$ be the map defined in \eqref{eqn:map_fn}. The map $\psi_n$ adjoint to 
	\[
		\begin{tikzcd}
			f_n^*\tau_F^n \colon \Sigma \C P^{n-1} \simeq S \C P^{n-1} \ar[r,"f_n"] & SU(n) \ar[r,"\tau_F^n"] & B\BUloc{d}  			
		\end{tikzcd}
	\]
	classifies the bundle $\frac{1}{d} F(H)$, where $H \to \C P^{n-1}$ denotes the tautological line bundle.
\end{corollary}
\begin{proof}
	This is now a direct consequence of Lemma~\ref{lem:classifying_map_plus}. Note that $F(H)$ is classified by a map $\C P^{n-1} \to \BUloc{d} \times \{1\} \subset \BUloc{d} \times \Z$ (see Remark~\ref{rem:Z-component}) and we have $[F(H)] = [\psi_n] \cdot [F(\C)]$ as classes in $[\C P^{n-1}, \BUloc{d} \times \Z]$. 
\end{proof}

With Cor.~\ref{cor:classifying_map_tensor} we now have all the tools in place to compute the indecomposable terms in the rational characteristic classes of $\tau^n_F$ for a polynomial exponential functor $F$ in terms of the generators of $H^{\rm odd}(SU(n),\Q)$. 

\begin{theorem} \label{thm:char_classes}
	Let $n \in \N$, let $F \colon \Vfin \to \Vfin$ be a polynomial exponential functor and let $(b_1, \dots, b_k) \in \N^k$ be the rescaled Thoma parameters of its $R$-matrix. The class $\delta^{F,n} \in H^{\text{odd}}(SU(n),\Q)$ obtained from the rationalisation of the twist $\tau_F^n  \colon SU(n) \to B\BUloc{d}$ as in \eqref{eqn:delta} takes the form
	\[
		\delta^{F,n} = \kappa_1 a_3 + \kappa_2 a_5 + \dots + \kappa_{n-1} a_{2n-1} + r \in H^{\text{odd}}(SU(n),\Q)
	\]
	where $\kappa_i$ is the $i$th coefficient in the Taylor expansion of 
	\[
		\kappa(x) = \sum_{j=1}^k \log \left( \frac{1+b_je^x}{1+b_j} \right)
	\]
	and $r \in H^{\text{odd}}(SU(n),\Q)$ only contains decomposable terms. 
\end{theorem}

\begin{proof}
Let $\tau_{F,\Q}^n \colon SU(n) \to (BBU)_{\Q}$ be the composition of $\tau_F^n$ with the natural map $B\BUloc{d} \to (BBU)_{\Q}$. As described in the paragraph after \eqref{eqn:rational_equiv} the class $\delta^{F,n}$ is represented by the map $B(\log\,\ch) \circ \tau_{F,\Q}^n$. By Lemma~\ref{lem:map_detects_coh_classes} the coefficient $\kappa_i$ agrees with the coefficient of $t^i$ of the class 
\[
	f_n^*\delta^{F,n} \in H^{\text{odd}}(S\C P^{n-1},\Q) \cong H^{\text{even}}(\C P^{n-1},\Q) \cong \Q[t]/(t^n)\ .
\]
By Cor.~\ref{cor:classifying_map_tensor} this class agrees with $\log\ch(\frac{1}{d}F(H))$. In particular, it only depends on the isomorphism class of $F(H)$. Let $F^{b_i} = F^{\C^{b_i}}$. We can now apply Thm.~\ref{thm:classification_part_1} to obtain 
\begin{align*}
	\log\ch(F(H)) &= \log\ch(F^{b_1}(H) \otimes \dots \otimes F^{b_k}(H)) \\
	&= \log \left(\ch(F^{b_1}(H)) \cdots \ch(F^{b_k}(H)) \right) = \sum_{i=1}^k \log \ch(F^{b_i}(H))
\end{align*}
Note that we have $d = (1 + b_1) \cdots (1+b_k)$. Thus, we are left with the computation of $\frac{1}{1+b_i} \ch(F^{b_i}(H))$. This turns out to be 
\[
	\frac{1}{1+b_i}\ch(F^{b_i}(H)) = \frac{1}{1+b_i}\left(\ch(\C \oplus \C^{b_i} \otimes \Lambda^1(H) \right) = \frac{1 + b_ie^t}{1+b_i}
\]
which proves the statement.
\end{proof}

\begin{corollary}
Let $F \colon \Vfin \to \Vfin$ be a polynomial exponential functor and let $(b_1, \dots, b_k)$ be the rescaled Thoma parameters of its associated $R$-matrix. For the characteristic classes $\delta^{F,2}$ and $\delta^{F,3}$ associated to the twists $\tau^F_2$ over $SU(2)$ and $\tau^F_3$ over $SU(3)$, respectively, we obtain 
\begin{align*}
	\delta^{F,2} &= \delta^{F,2}_3 = \sum_{i=1}^k \frac{b_i}{b_i+1}a_3 \ , \\
	\delta^{F,3} &= \delta^{F,3}_3 + \delta^{F,3}_5 = \sum_{i=1}^k \left(\frac{b_i}{b_i+1}a_3 + \frac{b_i}{2(b_i + 1)^2}a_5 \right)\ . 
\end{align*}
\end{corollary}

\begin{proof}
The rational cohomology of $SU(2)$ is $H^*(SU(2),\Q) \cong \Lambda_\Q[a_3]$. Since all products are trivial in this ring, there are no decomposable terms. Likewise, $H^*(SU(3),\Q) \cong \Lambda_\Q[a_3,a_5]$. Therefore all decomposable expressions live in degree $8$, which is even. 
\end{proof}

Finally, it is also possible to treat the classical twists obtained from the basic gerbe over $SU(n)$ using the techniques developed above as explained in the next theorem.

\begin{theorem} \label{thm:class_of_classical_twist}
Let $\tau_{\det}^n \colon SU(n) \to BBU_{\otimes}$ be the twist associated to the exponential functor $\extp^{\rm top}	 \colon \Viso \to \Viso$ (see Rem.~\ref{rem:ext_power} and the beginning of this section). The associated class $\delta^{\det,n} \in H^{\rm odd}(SU(n),\Q)$ satisfies 
\[
	\delta^{\det,n} = a_3\ .
\]
\end{theorem}

\begin{proof}
We proceed as in the proof of the previous theorem. For $F = \extp^{\rm top}$ we obtain $F(H) \cong H$ and $d = 1$. Therefore 
\[
	\log\ch(F(H)) = \log \ch(H) = \log \exp(t) = t\ .
\]
Therefore $\kappa_1 = 1$. But since $\tau^n_{\det}$ factors through $BBU(1)$, the class $\delta^{\det,n}$ can only be non-trivial in degree~$3$.
\end{proof}


\begin{remark} \label{rem:coadjoint_orbits}
	Observe that for a fixed suspension coordinate $t$ the image of $f_n$ is the adjoint orbit of $\omega(t)$ in $SU(n)$ consisting of all matrices with one eigenvalue $\exp(2\pi i t\tfrac{n-1}{n})$ and $(n-1)$ eigenvalues of the form $\exp(-2 \pi i t\tfrac{1}{n})$. This can be interpreted in terms of the decomposition of $SU(n)$ as a simplicial space as in \cite[Eq.~(18)]{paper:Meinrenken}:  The suspension coordinate runs along one of the edges in the $1$-skeleton of that decomposition. Using a different edge we obtain analogous maps
	\[
		SGr_k^n \to SU(n)
	\] 
	where $Gr_k^n$ denotes the Grassmann manifold of complex $k$-dim.\ subspaces in $\C^n$. We will leave it as an open question whether these maps yield more information about the rational cohomology classes of $\tau^n_F$ if an analogous analysis as the one above is performed with them.
\end{remark}

\subsection{From twists to Fell bundles}
We end this paper with an outline of the operator algebraic description of the twists $\tau_F^n$ in terms of Fell bundles. For reasons of brevity we will only present a sketch here. The proofs of the claims made in this section will be part of upcoming work (see \cite{paper:EvansPennigFellBundles}). A Fell bundle consists of a groupoid $\cG$ together with a Banach bundle $\pi \colon \cE \to \cG$ over the space of morphisms of $\cG$ equipped with the following additional structure: Let $\cG^{(2)} \subset \cG \times \cG$ be the subspace of composable arrows and let 
\[
	\cE^{(2)} = \{ (e_1,e_2) \in \cE \times \cE \ |\ (\pi(e_1),\pi(e_2)) \in \cG^{(2)}  \}\ .
\]
The Fell bundle $\cE$ carries a bilinear and associative multiplication map $\cE^{(2)} \to \cE$ denoted by a dot and an operator $\ast \colon \cE \to \cE$ inverting the underlying arrows of $\cG$. The norm on the fibres should be submultiplicative and satisfy the $C^*$-condition $\lVert e^* \cdot e \rVert = \lVert e \rVert^2$. As a consequence the fibres of $\cE$ over the objects of $\cG$ are $C^*$-algebras and provide bimodules for the corresponding source and target algebras over the non-identity fibres. If these bimodules are Morita equivalences, the Fell bundle is called saturated. For the full details of the definition we refer the reader to \cite{paper:BussMeyerZhu, paper:Kumjian}. In case $\cG = Y^{[2]}$ is the pair groupoid of a fibration $Y \to X$ over a topological space $X$ and $\cE \to \cG$ is a line bundle the construction boils down to that of a bundle gerbe in the sense of \cite{paper:MurrayStevenson}.

Our construction is now based on the same idea as in \cite{paper:MurrayStevenson}: For $g \in SU(n)$ let $\EV{g} \subset S^1$ be the set of distinct eigenvalues of $g$. Given $\lambda \in \EV{g}$ let $\Eig{g}{\lambda} \subset \C^n$ be the eigenspace of $g$ with respect to~$\lambda$. Fix a total order on $S^1 \setminus \{1\}$ and define
\[
	Y = \{ (g,z) \in SU(n) \times (S^1 \setminus \{1\})\ |\ z \notin \EV{g}  \}\ .
\]   
The fibre product $Y^{[2]} = \{ (g,z_1,z_2) \in SU(n) \times (S^1 \setminus \{1\})\ |\ z_1,z_2 \notin \EV{g}\}$ decomposes into three disjoint subspaces: the space $Y^{[2]}_+$ consisting of those triples $(g,z_1,z_2)$ with $z_1 < z_2$ and at least one eigenvalue of $g$ between $z_1$ and~$z_2$, the space $Y^{[2]}_0$ containing all triples with no eigenvalues between $z_1$ and~$z_2$ and $Y^{[2]}_-$ obtained by interchanging the role of $z_1,z_2$ in $Y^{[2]}_+$.

Let $F \colon \Viso \to \Viso$ be an exponential functor and let $M_F = \Endo{F(\C^n)}^{\otimes \infty}$ be the infinite tensor product of the full matrix algebra $\Endo{F(\C^n)}$. This is a uniformly hyperfinite $C^*$-algebra. In particular, it is strongly self-absorbing. Define $\cE^{+} \to Y^{[2]}_+$ fibrewise as follows:
\[
	\cE^{+}_{(g,z_1,z_2)} = \bigotimes_{z_1 < \lambda < z_2 \atop \lambda \in \EV{g}} F(\Eig{g}{\lambda}) \otimes M_F\ .
\]
Each fibre is a right Hilbert $M_F$-module where $M_F$ acts by right multiplication on the tensor factor $M_F$. Let $W = \Eig{g}{\lambda} \subset \C^n$. There are natural isomorphisms $F(\C^n) \cong F(W \oplus W^{\perp}) \cong F(W) \otimes F(W^{\perp})$ induced by the exponential structure of $F$. As a consequence the infinite UHF-algebra $M_F$ absorbs the $C^*$-algebra
\[
	\bigotimes_{z_1 < \lambda < z_2 \atop \lambda \in \EV{g}}\Endo{F(\Eig{g}{\lambda})}
\]
and there is an isomorphism $\mathcal{K}_{M_F}(\cE^{+}_{(g,z_1,z_2)}) \cong M_F$, where the left hand side denotes the compact $M_F$-linear operators on the Hilbert module $\cE^{+}_{(g,z_1,z_2)}$. This structure turns $\cE^{+}_{(g,z_1,z_2)}$ into an $M_F$-$M_F$-Morita equivalence. We have that $\cE^{+} \to Y^{[2]}_+$ is a Banach bundle and the isomorphisms described above can be chosen in such a way that they form a bilinear associative multiplication 
\[
	\cE^{+}_{(g,z_1,z_2)} \times \cE^{+}_{(g,z_2,z_3)} \to \cE^{+}_{(g,z_1,z_3)}
\]
for all $(g,z_1,z_2), (g,z_2,z_3) \in Y^{[2]}_+$. Moreover, if we define 
\[
	\cE^{0} = Y^{[2]}_0 \times M_F \qquad \text{and} \qquad \cE^{-}_{(g,z_1,z_2)} = \left(\cE^{+}_{(g,z_2,z_1)}\right)^{\rm op}\ ,
\]
then the multiplication extends to the disjoint union $\cE \to Y^{[2]}$, is still bilinear and associative and turns $\cE$ into a saturated Fell bundle over $\cG = Y^{[2]}$. Choosing a Haar system on $\cG$ we can now look at the $C^*$-algebra $C^*(\cE)$ of sections of $\cE$. Pulling back continuous functions from $SU(n)$ to $Y^{[2]}$ we see that each $f \in C(SU(n))$ acts as a mutiplier on $C^*(\cE)$ and as such provides a central element. This turns $C^*(\cE)$ into a continuous $C(SU(n))$-algebra, whose fibres are all isomorphic to $M_F \otimes \K$, where $\K$ denotes the compact operators. The algebra $C^*(\cE)$ satisfies the Fell condition stated in \cite[Def.~4.1]{paper:DadarlatPennigI}. Thus, by \cite[Thm.~4.2]{paper:DadarlatPennigI} it is isomorphic to the section algebra of a locally trivial bundle of $C^*$-algebras over $SU(n)$ with fibre $M_F \otimes \K$. But since  
\[
	[SU(n), B\Aut{M_F \otimes \K}] \cong \left[SU(n), B\BUloc{d}\right]
\]
by \cite{paper:DadarlatPennigII} (with $d = \dim(F(\C))$) the bundle obtained from $C^*(\cE)$ corresponds to a twist, which we claim agrees with $\tau^n_F$. The upshot of this construction is that the adjoint action of $SU(n)$ on itself lifts to an action of $SU(n)$ on $\cE$ in a canonical way, since $h \in SU(n)$ maps $\Eig{g}{\lambda}$ to $\Eig{hgh^{-1}}{\lambda}$ mimicking the equivariant bundle gerbe in the classical case \cite{paper:MurrayStevenson}.

\bibliographystyle{plain}
\bibliography{ExponentialFunctors}

\newpage
\appendix 
\section{An $\I$-monoid model for $\BUloc{d}$}
In this appendix we present a model for the infinite loop space $\BUloc{d}$ as a commutative $\I$-monoid. It provides an alternative to the model via permutative categories used in Sec.~\ref{sec:twists_via_mon_cats}. We include it here, since it might be of independent interest.

Let $\I$ be the category with objects the finite sets $\mathbf{n} = \{1,\dots,n\}$ (including the empty set $\mathbf{0}$) and injective maps between them as morphisms. This category is symmetric monoidal: The tensor product is given by concatenation $\mathbf{m} \sqcup \mathbf{n} = \{1, \dots, m+n \}$, where $\mathbf{m}$ corresponds to the first $m$ entries and $\mathbf{n}$ to the last $n$. The symmetry isomorphisms are the obvious shuffle maps $s_{m,n} \colon \mathbf{m} \sqcup \mathbf{n} \to \mathbf{n} \sqcup \mathbf{m}$. Let $\mathcal{T}op_{\ast}$ be the category of pointed topological spaces and basepoint preserving continuous maps.

An \emph{$\I$-space} is a functor $X \colon \I \to \mathcal{T}op_{\ast}$. An $\I$-monoid is an $\I$-space $X$ together with a natural transformation
\[
	\begin{tikzcd}
		\mu \colon X \times X \ar[r,Rightarrow] & X \circ \sqcup
	\end{tikzcd}
\]
given by maps $\mu_{m,n} \colon X(\mathbf{m}) \times X(\mathbf{n}) \to X(\mathbf{m} \sqcup \mathbf{n})$ that satisfy associativity and unitality conditions in the sense that the basepoint of $X(\mathbf{0})$ acts as a unit for $\mu$. An $\I$-monoid is called \emph{commutative} if the following diagram commutes
\[
	\begin{tikzcd}[column sep=1.8cm]
		X(\mathbf{m}) \times X(\mathbf{n}) \ar[d,"\text{tw}" left] \ar[r,"\mu_{m,n}"] & X(\mathbf{m} \sqcup \mathbf{n}) \ar[d,"X(s_{m,n})"] \\
		X(\mathbf{n}) \times X(\mathbf{m}) \ar[r,"\mu_{n,m}" below] & X(\mathbf{n} \sqcup \mathbf{m})
	\end{tikzcd}
\]
For details about this construction we refer the reader to \cite{paper:SagaveSchlichtkrull, paper:Schlichtkrull}. Let $W$ be a complex inner product space with $d = \dim(W)$ and define 
\[
	\Bte{W}(\mathbf{n}) = BU(W^{\otimes n})\ .
\] 
A permutation $\sigma \colon \mathbf{n} \to \mathbf{n}$ acts on $W^{\otimes n}$ by permuting the tensor factors. The adjoint action of the corresponding unitary yields $\Bte{W}(\sigma)$. For the proper inclusion $\iota \colon \mathbf{m} \to \mathbf{n}$ we define $\Bte{W}(\iota) \colon BU(W^{\otimes m}) \to BU(W^{\otimes n})$ to be the continuous map induced by the group homomorphism 
\[
	\begin{tikzcd}[column sep=1.8cm]
		U(W^{\otimes m}) \ar[r,"u \mapsto u \otimes 1"] & 
		U(W^{\otimes m} \otimes W^{\otimes (n-m)})) = U(W^{\otimes n}) 		
	\end{tikzcd}\ .
\]
Any morphism $f \colon \mathbf{m} \to \mathbf{n}$ can be written as $\sigma \circ \iota$, where $\sigma \colon \mathbf{n} \to \mathbf{n}$ is the permutation that maps $1,\dots, m$ to $f(1), \dots, f(m)$ and fills the gaps with the element $m+1, \dots, n$ in ascending order. It is straightforward to see that the definition $\Bte{W}(f) := \Bte{W}(\sigma) \circ \Bte{W}(\iota)$ indeed yields a functor. This fixes $\Bte{W} \colon \I \to \mathcal{T}op_{\ast}$ completely and gives it the structure of an $\I$-space. It can be equipped with an $\I$-monoid structure, where the multiplication
\[
	\mu_{m,n}^{\otimes} \colon \Bte{W}(\mathbf{m}) \times \Bte{W}(\mathbf{n}) \to \Bte{W}(\mathbf{m} \sqcup \mathbf{n})
\]
is induced by the tensor product of unitaries. With respect to these definitions $\Bte{W}$ is a commutative $\I$-monoid \cite{paper:SagaveSchlichtkrull, paper:Schlichtkrull}. However, it is not convergent. We will determine the homotopy type of $(\Bte{W})_{h\I}$ in the next lemma. Let $\mathcal{N}$ be the poset $\N$ considered as a category with a unique arrow from $m$ to $n$ if and only if $m \leq n$. Note that $\mathcal{N}$ embeds into $\I$ by sending $n \in \mathcal{N}$ to $\mathbf{n} \in \I$. and the arrow $m \to n$ in $\mathcal{N}$ to the embedding $\iota \colon \mathbf{m} \to \mathbf{n}$ in $\I$ as defined above. We define 
\[
	\left(\Bte{F}\right)_{\infty} = \hocolim_{\mathcal{N}}\left(\Bte{F}\right)\ .
\]

\begin{lemma} \label{lem:htpy_BtensorW}
	Let $W \in \obj{\Vfin}$ and let $d = \dim(W)$. We have 
	\[
		\left(\Bte{W}\right)_{\infty} \simeq \BUloc{d}
	\] 
	and the continuous map $\left(\Bte{W} \right)_{\infty} \to \left(\Bte{W}\right)_{h\I}$ is a homotopy equivalence.
\end{lemma}

\begin{proof}
	Since $\left(\Bte{W}\right)_{\infty}$ agrees up to homotopy equivalence with the path component labelled by $1 \in \Z$ of the space $\lvert \BUte{W} \rvert_{\infty}$ that we considered in Lemma~\ref{lem:htpy_type_BUF}, the first statement of the lemma is clear. 
	
	We will prove the second statement by showing that there exists a model for $BU(W^{\otimes n})$ which satisfies the conditions of \cite[Thm.~3.1 and Thm.~3.3]{paper:AdemGomezLindTillmann}. Combining these theorems we obtain that
	\[
		\left(\Bte{W} \right)_{\infty} \simeq \left(\Bte{W} \right)_{\infty}^+ \simeq \left(\Bte{W}\right)_{h\I}\ ,
	\]
	where we used that $\BUloc{d}$ is simply connected to obtain the first equivalence. Let $H$ be a separable infinite-dimensional Hilbert space, denote by $\K$ the compact operators on $H$ and 	let $e \in \K$ be a projection of rank~$1$. The unitary group $U(W \otimes H)$ endowed with the norm topology is contractible by Kuiper's theorem. Identify $W$ with the subspace $W \otimes eH \subset W \otimes H$. The space 
	\[
		U(W \otimes H)/U(W) \times U(W^\perp)
	\]
	is a model for $BU(W)$. The map $u \mapsto u(\id{W} \otimes e)u^*$ induces a homeomorphism
	\begin{equation} \label{eqn:homeo_proj}
		U(W \otimes H)/U(W) \times U(W^{\perp}) \cong \Proj{\id{W} \otimes e}{\text{End}(W) \otimes \K}		
	\end{equation}
	where the path component of $\id{W} \otimes e$ in the space of self-adjoint projections in the $C^*$-algebra $\text{End}(W) \otimes \K$ is equipped with the norm topology. This space is a Banach manifold. Let $q = \id{W} \otimes e$ and 
	\[
		\mathcal{B}_{\otimes}^{\mathcal{P}}(\mathbf{n}) = \Proj{q^{\otimes n}}{(\text{End}(W) \otimes \K)^{\otimes n}}\ . 
	\]
	Note that $\mathcal{B}_{\otimes}^{\mathcal{P}}(\mathbf{0}) = \{\id{\C}\}$. A permutation $\sigma \colon \mathbf{n} \to \mathbf{n}$ acts on $(W \otimes H)^{\otimes n}$ by permuting the tensor factors. Thus, it acts on $\mathcal{B}_{\otimes}^{\mathcal{P}}(\mathbf{n})$ via the corresponding adjoint action. This defines $\mathcal{B}_{\otimes}^{\mathcal{P}}(\sigma)$. Let $\iota \colon \mathbf{m} \to \mathbf{n}$ be the inclusion that identifies $\mathbf{m}$ with the first $m$ entries of~$\mathbf{n}$. Then $\mathcal{B}_{\otimes}^{\mathcal{P}}(\iota)$ maps $p \in \mathcal{B}_{\otimes}^{\mathcal{P}}(\mathbf{m})$ to $p \otimes q^{\otimes (n - m)} \in \mathcal{B}_{\otimes}^{\mathcal{P}}(\mathbf{n})$. As in the construction above this turns $\mathcal{B}_{\otimes}^{\mathcal{P}}$ into an $\I$-space. It is a commutative $\I$-monoid with respect to the multiplication
	\[
		\mu_{m,n} \colon \mathcal{B}_{\otimes}^{\mathcal{P}}(\mathbf{m}) \times \mathcal{B}_{\otimes}^{\mathcal{P}}(\mathbf{n}) \to \mathcal{B}_{\otimes}^{\mathcal{P}}(\mathbf{m} \sqcup \mathbf{n}) \quad , \quad \mu_{m,n}(p_1,p_2) = p_1 \otimes p_2\ .
	\]
	For $f \colon \mathbf{m} \to \mathbf{n}$ the induced map $\mathcal{B}_{\otimes}^{\mathcal{P}}(f)$ embeds $\mathcal{B}_{\otimes}^{\mathcal{P}}(\mathbf{m})$ into $\mathcal{B}_{\otimes}^{\mathcal{P}}(\mathbf{n})$ as a closed submanifold. All of these maps are cofibrations by \cite[Thm.~7]{paper:Palais}. Moreover, if $p_1 = p_2 \otimes e \in \K \otimes \K$ is a projection it follows that $p_2$ has to be a projection as well. Therefore, $\mu_{m,n}(p_1,p_2)$ is in the image of $\mathcal{B}_{\otimes}^{\mathcal{P}}(f)$ for a strict inclusion $f$ if and only if $p_1$ or $p_2$ is. Thus, $\mathcal{B}_{\otimes}^{\mathcal{P}}$ satisfies the hypotheses of \cite[Thm.~3.3]{paper:AdemGomezLindTillmann}. 
	
	We will now construct an $\I$-monoid $\mathcal{B}_{\otimes}^{EU}$ together with $\I$-monoid morphisms $\Bte{W} \to \mathcal{B}_{\otimes}^{EU}$ and $\mathcal{B}_{\otimes}^{\mathcal{P}} \to \mathcal{B}_{\otimes}^{EU}$ that are object-wise homotopy equivalences. Let
	\[
		\mathcal{B}_{\otimes}^{EU}(\mathbf{n}) = EU((W \otimes H)^{\otimes n})/U(W^{\otimes n}) \times U((W^{\otimes n})^{\perp})\ ,
	\]
	 where $W^{\otimes n}$ embeds as a direct summand into $(W \otimes H)^{\otimes n}$ via $e$ like above. The adjoint action of the permutation of tensor factors defines $\mathcal{B}_{\otimes}^{EU}(\sigma)$ for all $\sigma \colon \mathbf{n} \to \mathbf{n}$ and the group homomorphism $U((W \otimes H)^{m}) \to U((W \otimes H)^{n})$ given by $u \mapsto u \otimes \id{(W \otimes H)^{\otimes (n-m)}}$ induces $\mathcal{B}_{\otimes}^{EU}(\iota)$. The $\I$-monoid structure is induced by the homomorphism
	 \[
	 	U((W\otimes H)^{\otimes n}) \times U((W\otimes H)^{\otimes m}) \to U((W\otimes H)^{\otimes (n+m)}) 
	 \]
	 given by taking the tensor product of two unitaries. Note in particular, that this maps the subgroup $U(W^{\otimes n}) \times U((W^{\otimes n})^{\perp}) \times U(W^{\otimes m}) \times U((W^{\otimes m})^{\perp})$ into $U(W^{\otimes (n +m)}) \times U((W^{\otimes (n+m)})^{\perp})$, since these are the stabilisers with respect to the adjoint action of the projections $q^{\otimes n}$, $q^{\otimes m}$ and $q^{\otimes (n+m)}$, respectively. This structure turns $\mathcal{B}_{\otimes}^{EU}$ into a commutative $\I$-monoid. Using the homeomorphism \eqref{eqn:homeo_proj} each space $\mathcal{B}_{\otimes}^{EU}(\mathbf{n})$ yields a locally trivial bundle 
	 \[
	 	\mathcal{B}^\mathcal{P}_{\otimes}(\mathbf{n}) \to \mathcal{B}^{EU}_{\otimes}(\mathbf{n}) \to BU((W \otimes H)^{\otimes n}) 
	 \] 
	 where the first map is the inclusion of the fibre over the basepoint. It is straightforward to check that the homeomorphism \eqref{eqn:homeo_proj} transforms the commutative $\I$-monoid structure on the fibre of $\mathcal{B}^{EU}_{\otimes}(\mathbf{n}) \to BU((W \otimes H)^{\otimes n})$ to the one on $\mathcal{B}^\mathcal{P}_{\otimes}(\mathbf{n})$. Moreover, the group $U((W \otimes H)^{\otimes n})$ is contractible (again by Kuiper's theorem). Therefore the map $\mathcal{B}^\mathcal{P}_{\otimes}(\mathbf{n}) \to \mathcal{B}^{EU}_{\otimes}(\mathbf{n})$ is a homotopy equivalence for all $\mathbf{n} \in \obj{\I}$. The embedding $U(W^{\otimes n}) \to U((W \otimes H)^{\otimes n})$ induces a map $EU(W^{\otimes n}) \to EU((W\otimes H)^{\otimes n})$, which yields a homotopy equivalence
	 \[
	 	EU(W^{\otimes n})/U(W^{\otimes n}) \to EU((W \otimes H)^{\otimes n})/U(W^{\otimes n}) \times U((W^{\otimes n})^{\perp})
	 \]
	 Therefore we obtain homotopy equivalences
	 \[
	 	\Bte{W}(\mathbf{n}) = BU(W^{\otimes n}) \to \mathcal{B}^{EU}_{\otimes}(\mathbf{n})\ ,
	 \]
	 for all $\mathbf{n} \in \obj{\I}$. These are compatible with all maps defining the commutative $\I$-monoid structure.
\end{proof}

\end{document}